\RequirePackage{fix-cm}
\documentclass[11pt,reqno]{amsart}
\usepackage{amsmath,amssymb,amsthm}
\usepackage{amsfonts}
\usepackage{mathtools}
\usepackage{enumitem}

\usepackage{color}
\usepackage[pdftex,breaklinks]{hyperref}
\hypersetup{colorlinks=true,linkcolor=blue,citecolor=red}

\numberwithin{equation}{section}

\theoremstyle{plain}
\newtheorem{theorem}{Theorem}[section]
\newtheorem{lemma}[theorem]{Lemma}

\newtheorem{proposition}[theorem]{Proposition}
\newtheorem{corollary}[theorem]{Corollary}

\newtheorem*{theorem-non}{Theorem}

\theoremstyle{definition}
\newtheorem{definition}[theorem]{Definition}

\newtheorem{remark}[theorem]{Remark}

\theoremstyle{definition}
\newtheorem{hypothesis}[theorem]{Hypothesis}

\theoremstyle{definition}

\newcommand{\be}{\begin{equation}}
\newcommand{\ee}{\end{equation}}
\newcommand{\n}{\noindent}

\DeclareMathOperator{\loc}{loc}
\DeclareMathOperator{\SO}{SO}
\DeclareMathOperator{\BMO}{BMO}
\DeclareMathOperator{\dist}{distance}

\DeclareMathOperator{\Div}{Div}
\DeclareMathOperator{\tr}{trace}

\DeclareMathOperator{\Var}{Var}

\DeclareMathOperator{\Lin}{Lin}
\DeclareMathOperator{\Psymn}{Psym_\mathit{n}}
\DeclareMathOperator{\Symn}{Sym_\mathit{n}}
\DeclareMathOperator{\Def}{Def}
\DeclareMathOperator{\AD}{AD}

\newcommand{\tsn}[1]{{[\kern-0.3ex] #1 [\kern-0.3ex]}}

\newcommand{\Ksubset}{{\subset\!\subset}}

\newcommand{\oo}{_{\mathrm o}}

\newcommand{\grad}{\nabla}

\newcommand{\id}{{\mathbf{id}}}

\newcommand{\ku}{{\tau_u}}
\newcommand{\kv}{{\tau_v}}

\newcommand{\bee}{{\mathbf e}}

\newcommand{\bL}{{\mathbf L}}
\newcommand{\bB}{{\mathbf B}}
\newcommand{\bD}{{\mathbf P}}
\newcommand{\bC}{{\mathbf C}}

\newcommand{\bV}{{\mathbf V}}
\newcommand{\bT}{{\mathbf T}}
\newcommand{\bU}{{\mathbf U}}

\newcommand{\bE}{{\mathbf E}}
\newcommand{\bH}{{\mathbf H}}

\newcommand{\bI}{{\mathbf I}}
\newcommand{\bF}{{\mathbf F}}

\newcommand{\bK}{{\mathbf K}}
\newcommand{\bu}{{\mathbf u}}

\newcommand{\bv}{{\mathbf v}}
\newcommand{\bx}{{\mathbf x}}
\newcommand{\bd}{{\mathbf d}}
\newcommand{\by}{{\mathbf y}}
\newcommand{\bz}{{\mathbf z}}
\newcommand{\bn}{{\mathbf n}}

\newcommand{\bw}{{\mathbf w}}
\newcommand{\bQ}{{\mathbf Q}}
\newcommand{\bS}{{\mathbf S}}
\newcommand{\bs}{{\mathbf s}}
\newcommand{\ba}{{\mathbf a}}
\newcommand{\bb}{{\mathbf b}}

\newcommand{\bR}{{\mathbf R}}
\newcommand{\bM}{{\mathbf M}}

\newcommand{\bc}{{\mathbf c}}

\newcommand{\bm}{{\mathbf m}}

\newcommand{\sB}{{\mathcal B}}
\newcommand{\sS}{{\mathcal S}}
\newcommand{\sD}{{\mathcal D}}

\newcommand{\sH}{{\mathcal H}}
\newcommand{\sK}{{\mathcal K}}

\newcommand{\sN}{{\mathcal N}}

\newcommand{\mi}{{\text{-}}}

\newcommand{\CC}{{\mathbb C}}

\newcommand{\R}{{\mathbb R}}
\newcommand{\A}{{\mathbb A}}

\newcommand{\Mn}{{\mathbb M}^{n\times n}}

\newcommand{\Mnp}{{\mathbb M}_+^{n\times n}}

\newcommand{\bue}{{{\mathbf u}\e}}

\newcommand{\Om}{{\Omega}}

\newcommand{\Omc}{{\overline{\Omega}}}

\newcommand{\E}{{\mathcal E}}

\newcommand{\e}{_{\mathrm e}}

\newcommand{\dd}{{\mathrm d}}
\newcommand{\DD}{{{\mathrm D}}}

\newcommand{\rmT}{{\mathrm T}}

\newcommand{\altd}{{\delta}}

\newcommand{\grads}{{\grad}_{\!\mathrm s}}

\def\Xint#1{\mathchoice
   {\XXint\displaystyle\textstyle{#1}}
   {\XXint\textstyle\scriptstyle{#1}}
   {\XXint\scriptstyle\scriptscriptstyle{#1}}
   {\XXint\scriptscriptstyle\scriptscriptstyle{#1}}
   \!\int}
\def\XXint#1#2#3{{\setbox0=\hbox{$#1{#2#3}{\int}$}
     \vcenter{\hbox{$#2#3$}}\kern-.5\wd0}}

\def\dashint{\Xint-}

\newlength{\extramargin}
\begin{document}

\title[BMO and Elasticity]
{BMO and Elasticity: Korn's Inequality; Local Uniqueness in Tension}


\author[D. E. Spector]{Daniel E. Spector}
\address{Okinawa Institute of Science and Technology Graduate University,
Nonlinear Analysis Unit, 1919--1 Tancha, Onna-son, Kunigami-gun,
Okinawa, Japan}
\email{daniel.spector@oist.jp}

\thanks{}

\author[S. J. Spector]{Scott J. Spector}
\address{Department of Mathematics,
Southern Illinois University, Carbondale, IL 62901, USA}
\email{sspector@siu.edu}
\thanks{}


\date{3 April 2020}
\dedicatory{}

%
%



\begin{abstract} In this manuscript two $\BMO$ estimates are
obtained, one for Linear Elasticity and one for Nonlinear Elasticity.
It is first shown that the $\BMO$-seminorm of the gradient of a
vector-valued mapping is bounded above by a constant times the
$\BMO$-seminorm of the
symmetric part of its gradient, that is, a Korn inequality in $\BMO$.
The uniqueness of equilibrium for a
finite deformation whose principal stresses are everywhere
nonnegative is then considered.
It is shown that when the second variation of the energy,
when considered as a function of the strain,
is uniformly positive definite at such an
equilibrium solution, then there is a
$\BMO$-neighborhood in strain space where there are no other
equilibrium solutions.
\end{abstract}


\maketitle
\setlength{\parskip}{.5em}
\setlength{\parindent}{2em}
\baselineskip=15pt


\section{Introduction}\label{sec:Intro}


In 1972 Fritz John~\cite{Jo72} published a uniqueness theorem for Nonlinear
Elasticity that, until recently, was the only result of its kind. He showed
that, given a stress-free reference configuration whose elasticity tensor
is uniformly positive definite, there is an $L^\infty$-neighborhood of the
reference configuration in \emph{the space of strains}, rather than the
space of deformation gradients, in which
there is at most one smooth solution of the equations of equilibrium for the
\emph{pure-displacement problem} for a hyperelastic body.
His proof made use of the space $\BMO$, Bounded Mean Oscillation, a
space that John \& Nirenberg~\cite{JN61} had invented some ten years
earlier in order to better analyze problems in Elasticity.

Subsequently, although our understanding of the space $\BMO$ and its
applicability to systems of partial differential equations have advanced
significantly, the original goal of making use of $\BMO$ in problems of
Elasticity has not progressed.  Recently, the authors~\cite{SS19}
extended John's uniqueness result to include the \emph{mixed problem}.
In particular we showed that, given a smooth
equilibrium solution $\bu\e$  at which the second variation of
the energy is uniformly positive, there are no other equilibrium
solutions $\bv\e$ for which the
difference of the
right Cauchy-Green \emph{strain tensors}:
\be\label{eqn:strain-diff-zero-1}
(\grad\bu\e)^\rmT\grad\bu\e- (\grad\bv\e)^\rmT\grad\bv\e
\ee
\n is small in $L^\infty$.  Here  $\grad\bu$ denotes the deformation
gradient: an $n$ by $n$ matrix of partial derivatives of the components
of the deformation $\bu:\Om\to\R^n$, $(\grad\bu)^\rmT$ denotes the
transpose of $\grad\bu$, and we identify the body with the region
$\Omc\subset\R^n$ that it occupies in a fixed reference configuration.


In this manuscript we extend the results obtained in \cite{Jo72,SS19}.
We note that when $\Om$ has sufficiently smooth boundary
(Lipschitz suffices),
the space $\BMO(\Om)$ is a Banach space that is between $L^\infty$
and all of the other $L^p$-spaces,
that is, for all $p\in[1,\infty)$,
\[
L^\infty(\Om) \subset \BMO(\Om) \subset L^p(\Om).
\]
\n  Specifically,
$\tsn{\cdot}_{\BMO(\Om)}\le 2 \|\cdot\|_{L^\infty(\Om)}$ and hence an
$\varepsilon$-neighborhood in $\BMO$ is potentially much larger than
an $\varepsilon$-neighborhood in $L^\infty$.  Here
$\tsn{\cdot}$ denotes the standard seminorm on $\BMO(\Om)$
(see \eqref{eqn:BMO-V}).

We show, in particular,  that the $L^\infty$-neighborhood in which there is
at most one solution can be enlarged to a neighborhood in $\BMO$ for
both the displacement and the mixed problem \emph{provided}
the equilibrium solution $\bue$ has
nonnegative principal stresses everywhere.   Thus, in this case
the strain difference in \eqref{eqn:strain-diff-zero-1} need no longer
be uniformly small, but instead it need only be small in the space
$\BMO(\Om)$.

There are similar interesting results in the Calculus of
Variations literature.  Kristensen \& Taheri~\cite[Section~6]{KT03} and
Campos Cordero~\cite[Section~4]{Ca17} (see, also, Firoozye~\cite{Fi92})
have shown
that, for the Dirichlet problem, if $\bue$ is a Lipschitz-continuous
solution of the equilibrium equations at which the second variation
is uniformly positive, then
there is a neighborhood of $\grad\bue$ in $\BMO$ in which all
Lipschitz mappings have energy that is greater than or equal to
the energy of $\bue$.
We note that the assumptions in \cite{Ca17}, in particular, are
incompatible with the blowup of the energy as the Jacobian goes to zero.
Recently \cite{SS20-2} we have extended the results in
\cite[Section~4]{Ca17} to include the Neumann and mixed problems.
Although our proofs are not applicable to elasticity,
we have shown that given
a Lipschitz-continuous solution $\bue$ of the equilibrium equations
at which the second variation is uniformly positive, there is a
neighborhood of $\grad\bue$ in $\BMO$ in which all mappings $\bv$ in the
Sobolev space $W^{1,1}(\Om;\R^n)$ with $\grad\bv\in\BMO(\Om)$ have energy
that is \emph{strictly greater} than the energy of $\bue$.


We herein also establish a version of Korn's inequality
for $\BMO$.
It is well-known (see, e.g., \cite{ADM06,GS86,JK17,Ti72}) that,
for all $p\in(1,\infty)$, a generalized Korn inequality is valid,
that is, there is a constant $K=K(p)=K(p,n,\Om)$ such that
\[
\int_\Om \big|\grad\bw(\bx)\big|^p\,\dd\bx
\le
K(p) \int_\Om \big|\grad\bw(\bx)+[\grad\bw(\bx)]^\rmT\big|^p\,\dd\bx
\]
\n for all $\bw\in W^{1,p}(\Om;\R^n)$ that
satisfy a suitable constraint that eliminates infinitesimal
rotations (e.g., $\bw=\mathbf{0}$ on
$\sD\subset\partial\Om$).
We show that there exists
a constant $\sK=\sK(n)$ such that, for every nonempty, bounded
open set $U\subset\R^n$,
\be\label{eqn:BMO-Korn-intro}
\tsn{\grad\bw}_{\BMO(U)}
\le
\sK\tsn{\grad\bw+(\grad\bw)^\rmT}_{\BMO(U)},
\ee
\n  for every
$\bw\in W_{\loc}^{1,1}(U;\R^n)$ with $\grad\bw\in \BMO(U)$.  Note that,
unlike the standard Korn inequalities, which are only valid for John
domains (see \cite{JK17}) and for which the Korn constant depends on
the domain, \eqref{eqn:BMO-Korn-intro} is valid for all bounded
open sets $U$ with a constant that is independent of $U$.  (The lack of
a constraint to eliminate infinitesimal rotations is due to the nature of
the $\BMO$-seminorm.  See \eqref{eqn:BMO-V-norm}  and  \eqref{eqn:Korn2}.)


Before we present a more detailed description of our results,
we note that there is a long history of both nonuniqueness,
e.g., buckling \cite{SS08}, and uniqueness results in
nonlinear elasticity.  Rather than providing details here we
instead refer the reader to the introductions of two recent
papers concerning uniqueness \cite{SS18,SS19}. These papers
also discuss interesting possible extensions of such
results:  the pure-traction problem, incompressible
materials, and
live loading, none of which are considered in this manuscript.

We begin in Section~\ref{sec:prelim} with our notations. In
Section~\ref{sec:RGRSM} we then present certain consequences of the
Geometric Rigidity theory of Friesecke, James, \&  M{\"u}ller~\cite{FJM02}
(see, also, Conti \&  Schweizer~\cite{CS06} and Kohn~\cite{Ko82}) that
are useful in our work.  In
particular, a result of Lorent~\cite{Lo13} as well as a result of Ciarlet
\& Mardare~\cite{CM15} give conditions under which the equality of two
strains, $(\grad\bu)^\rmT\grad\bu\equiv(\grad\bv)^\rmT\grad\bv$,
yields the equality of the underlying deformations: $\bu\equiv\bv$.
(This need not be true without further assumptions, even if $\bu=\bv$ on
$\partial\Om$).


After reviewing certain standard properties of the space
$\BMO$, we then present, in Section~\ref{sec:maximal}, theorems
from Harmonic Analysis that we have found
useful in this work.  Of particular consequence is a result from
\cite{SS19}:  If $\Om$ is a Lipschitz domain and
$1\le p<q<\infty$, then there is a constant $C=C(p,q,\Om)$ such that,
for all $\psi\in\BMO(\Om)$,
\be\label{eqn:interp-into}
\int_\Om |\psi(\bx)|^q\,\dd\bx
\le
C
\big(\tsn{\psi}_{\BMO(\Om)}
+
\big| \langle\psi\rangle_{\Om}  \big|\big)^{q-p}
\int_\Om |\psi(\bx)|^p\,\dd\bx,
\ee
\n where  $\langle\psi\rangle_{\Om}$ denotes
the average value of the function $\psi$ on $\Om$.
This \emph{interpolation inequality} has a number of important
consequences.  Specifically,
we show that it implies that a result that John \& Nirenberg~\cite{JN61}
established for cubes is in fact valid for every nonempty, bounded,
open region $V\subset\R^n$:  For all $q\in(1,\infty)$  there exists a
constant $C=C(q)$ such that
\be\label{eqn:to-get-Korn}
\tsn{\phi}_{\BMO(V)}
\le
\sup_{Q\Ksubset V}\bigg(\,
\dashint_Q |\phi(\bx)-\langle\phi\rangle_Q|^q\,\dd\bx\bigg)^{\!\!1/q}
\le
C(q)\tsn{\phi}_{\BMO(V)},
\ee
\n for all $\phi\in\BMO(V)$, where the
supremum is taken over all cubes $Q$  that are compactly
supported in $V$
and have faces that are parallel to the coordinate planes.
(If $q=1$ the quantity in the center of
inequality \eqref{eqn:to-get-Korn} is equal to the $\BMO$-seminorm
of $\phi$.)
In Section~\ref{sec:Korn} we then make use of \eqref{eqn:to-get-Korn}
together with a version of Korn's inequality due to Diening,
R$\overset{_\circ}{\mathrm{u}}$\v zi\v cka, \& Schumacher~\cite{DRS10}
to establish Korn's inequality in $\BMO$, that is,
\eqref{eqn:BMO-Korn-intro}.


In Section~\ref{sec:NLE} we introduce our hypotheses on a compressible,
nonlinearly
hyperelastic body where the stored-energy density $\sigma$ depends on the
material point $\bx$ and the right Cauchy-Green strain tensor
$\bC_{\bu}(\bx)=[\grad\bu(\bx)]^\rmT\grad\bu(\bx)$.
Thus, in the absence of body forces and surface tractions, the
total energy of a deformation $\bu:\overline{\Omega}\to\R^n$,
which satisfies $\bu=\bd$ on $\sD\subset\partial\Om$, is given by
\[
\E(\bu)=\int_\Om \sigma\big(\bx,\bC_\bu(\bx)\big)\;\!\dd\bx.
\]

The second variation of $\E$ evaluated at a solution of the corresponding
equilibrium equations $\bue$ is then equal to
\[
\begin{aligned}
\delta^2\E(\bue)&[\bw,\bw]
= \int_\Omega \bK\big(\bx,\bC_\bue(\bx)\big):
\big[(\grad\bw)^\rmT\grad\bw\big]\dd\bx\\
&+\tfrac14\int_\Omega \bE(\bx):
\CC\big(\bx,\bC_\bue(\bx)\big)
\big[\bE(\bx)\big]\dd\bx,
\end{aligned}
\]
\n where $\bK=2\frac{\partial}{\partial\bC}\sigma(\bx,\bC)$  denotes
the (second) Piola-Kirchhoff stress tensor,
$\CC=4\frac{\partial^2}{\partial\bC^2}\sigma(\bx,\bC)$ denotes the
elasticity tensor,
$\bE=(\grad\bue)^\rmT\grad\bw+(\grad\bw)^\rmT\grad\bue$, and
$\bw\in W^{1,2}(\Om;\R^n)$ satisfies $\bw=\mathbf{0}$ on
$\sD\subset\partial\Om$.

If $\bK$ is positive semi-definite, equivalently, the principal
stresses are nonnegative, and $\CC$ is uniformly positive definite,
then $\delta^2\E(\bue)$ is uniformly positive.
Standard techniques (see, e.g., the introduction to
\cite{SS19}), which are usually applied in
the space of deformation gradients, make use of Taylor's theorem
to deduce that there is then an
$L^\infty$ neighborhood of $\bC_\bue$ in strain space
(see \eqref{eqn:strain-diff-zero-1}) in which there are
no other solutions of the equilibrium equations.  A refinement
of this argument, which is due to
John~\cite[pp.~624--625]{Jo72} (again, see the introduction to
\cite{SS19}),  makes use of
\eqref{eqn:interp-into} to enlarge the set in which
there are no other solutions to a neighborhood of $\bC_\bue$ in
the space $\BMO$.  We present the details of this argument in
Section~\ref{sec:unique-BMO+L1} of this manuscript.   We also note,
in Section~\ref{sec:Def-SS}, how these results simplify when one
of the two right Cauchy-Green strain tensors
is in an $L^\infty$-neighborhood of the reference
configuration.  Finally, in Section~\ref{sec:RCatE}, we present
further simplifications that occur when the reference configuration
is itself at equilibrium.


\section{Preliminaries}\label{sec:prelim}

For any domain (nonempty, connected, open set)
$U\subset\R^n$, $n\ge2$, we denote by
$L^p(U)$, $p\in[1,\infty)$, the space
of real-valued Lebesgue measurable functions $\psi$
whose $L^p$-norm is finite:
\[
||\psi ||^p_{p,U} := \int_U |\psi(\bx)|^p\,\dd\bx < \infty.
\]
\n $L^\infty(U)$ will denote those Lebesgue measurable
functions whose essential supremum is finite.
$L^1_{\loc}(U)$ will consist of those Lebesgue measurable
functions that are integrable on every compact subset of $U$.
 We shall write $C(U;\R^n)$ for the
set of continuous
functions  $\bu:U\to\R^n$, while $C^1(\overline{U};\R^n)$ will denote
those continuous functions $\bu:\overline{U}\to\R^n$ whose classical
derivative exists on $U$ and
has an extension that is continuous on $\overline{U}$, where
$\overline U$ denotes the closure of $U$.

We shall write $\Om\subset\R^n$, $n\ge2$,
to denote a \emph{Lipschitz domain}, that is, a
bounded domain whose boundary $\partial \Om$ is
(strongly) Lipschitz.
(See, e.g., \cite[p.~127]{EG92}, \cite[p.~72]{Mo66}, or
\cite[Definition~2.5]{HMT07}.)   Essentially, a bounded domain is
Lipschitz if, in a neighborhood of every boundary point,
the boundary is the graph of a
Lipschitz-continuous function and the domain is on ``one side''
of this graph.

For $1\le p\le \infty$,
$W^{1,p}(\Omega;\R^N)$ will denote the usual
Sobolev space of (Lebesgue) measurable
(vector-valued) functions $\bu\in L^p(\Omega;\R^N)$ whose
distributional gradient $\grad\bu$ is also contained in $L^p$.
If $\phi\in W^{1,p}(\Omega):=W^{1,p}(\Omega;\R)$ we shall denote
its  $W^{1,p}$-norm by\footnote{Since $\Om$ is a
Lipschitz domain, every $\phi\in W^{1,\infty}(\Omega)$ has a
representative that is Lipschitz continuous.}
\[
\begin{aligned}
||\phi||_{W^{1,p}(\Omega)}
&:=
\Big(||\phi||^p_{p,\Omega}
+||\grad\phi||^p_{p,\Omega}\Big)^{\!1/p},
\quad
1\le p <\infty,\\[2pt]
||\phi||_{W^{1,\infty}(\Omega)}
&
:= \max\{||\phi||_{\infty,\Omega}, ||\grad\phi||_{\infty,\Omega}\},
\quad
 p =\infty.
\end{aligned}
\]
\n We shall write $W^{1,p}_0(\Omega;\R^N)$ for the subspace of
$\bu\in W^{1,p}(\Omega;\R^N)$ that satisfy $\bu=\mathbf{0}$ on
$\partial \Omega$ (in the sense of trace).  $W_{\loc}^{1,p}(U;\R^N)$
will denote the set of $\bu\in W^{1,p}(V;\R^N)$
for every domain $V\Ksubset U$,
where we write $V\Ksubset U$ provided
that $V\subset K_V \subset U$ for some compact set $K_V$.

We shall write $\Mn$ for the (vector) space of $n$ by $n$ matrices with
real entries.
Given an orthonormal basis $\bee_i$, $i=1,2,\dots,n$, for $\R^n$ we
write $a_i=\ba\cdot\bee_i$ for $\ba\in\R^n$ and
$F_{ij}=\bee_i\cdot\bF\bee_j$ for $\bF\in\Mn$.
The set of symmetric and positive-definite symmetric matrices in $\Mn$
shall be denoted by
\[
\begin{aligned}
\Symn&:=\{\bE\in\Mn: \bE^\rmT=\bE\},\\
\Psymn&:=\{\bE\in\Symn:  \ba\cdot\bE\ba>0 \text{ for all }
\ba\in\R^n \text{ with } \ba\ne\mathbf{0}\},
\end{aligned}
\]
\n respectively, where $\bH^\rmT$ denotes the transpose of $\bH\in\Mn$.
We write $\bH:\bK:=\tr (\bH\bK^\rmT)$ for the inner product
of $\bH,\bK\in\Mn$. The norm of
$\bH\in\Mn$ is then given by $|\bH|:=\sqrt{\bH:\bH\,}$.  We write
\[
\SO(n):=\{\bQ\in\Mn: \bQ^\rmT\bQ=\bQ\bQ^\rmT=\bI,\ \det\bQ=1\}
\]
\n for the group of rotations, where $\bI$ denotes the identity matrix
and $\det\bF$ denotes the determinant of $\bF\in\Mn$.

\subsection{Strains and Geometric Rigidity}\label{sec:RGRSM}


Fix $p\ge1$. Given a mapping $\bu\in W^{1,p}(\Om;\R^n)$ we define
the \emph{right Cauchy-Green strain tensor}
$\bC_\bu\in L^{p/2}(\Om;\Symn)$ corresponding to $\bu$ by
\be\label{eqn:RCGST}
\bC_{\bu}:= (\grad\bu)^\rmT\grad\bu.
\ee
\n This tensor can be used to measure the change in the length of the
image of a curve in $\Om$ after it is deformed by $\bu$ (see, e.g.,
\cite[\S1.8]{Ci1988} or \cite[\S7.2]{GFA10}).


In \cite{FJM02} Friesecke, James, \&  M{\"u}ller (see, also, Conti \&
Schweizer~\cite{CS06}) establish a Geometric-Rigidity result that implies
that the distance (in $L^1$) from  $\bC_{\bu}$ to the identity matrix
yields, up to a multiplicative constant, an upper bound for the distance
(in $L^2$) from $\grad\bu$ to some particular rotation $\bQ_\bu\in\SO(n)$.
We shall make use of two interesting consequences of this result. The
first is a theorem of Lorent~\cite[Theorem~1]{Lo13} that establishes
conditions under which two mappings with the same strain tensor are
related by a rigid deformation:
\begin{proposition}\label{prop:Lo13}
Let
$\bv\in W^{1,1}(\Om;\R^n)$ satisfy $\det\grad\bv>0~a.e.$
Suppose that $\bu\in W^{1,n}(\Om;\R^n)$ satisfies
$\det\grad\bu>0~a.e.$, $\bC_\bu=\bC_\bv~a.e.$, and
\[
|\grad\bu(\bx)|^n\le K(\bx) \det\grad\bu(\bx) \text{ for } a.e.~\bx\in\Om,
\]
\n where\footnote{Lorent shows that, when $n=2$, $K\in L^1(\Om)$ suffices
and, when $n\ge3$, $K\in L^p(\Om)$ with $p>n-1$ suffices.}
$K\in L^n(\Om)$.  Then there exists a rotation $\bR\in\SO(n)$ such that
\[
\grad\bv(\bx) =\bR \grad\bu(\bx) \text{ for } a.e.~\bx\in\Om.
\]
\end{proposition}

Ciarlet \& Mardare have established a number of results that
bound the distance between two mappings in a Sobolev space
by a function of the distance between their right Cauchy-Green strain
tensors in a corresponding Lebesgue space.  The particular result
we shall employ is \cite[Theorem~3]{CM15}:
\begin{proposition}\label{prop:CM15}
Fix $p\in(1,\infty)$ and $q\in[r,p]$,
where $r:=\max\{1,p/2\}$.  Let $\bv\in C^1(\Omc;\R^n)$
satisfy $\det\grad\bv>0$ in $\Omc$.  Suppose that
$\sD\subset\partial \Om$ is nonempty and relatively open.
Then there exists a
constant $C_M=C_M(p,q,\bv,\Om,\sD)>0$ such that
\[
\int_\Om\big|\bC_\bu
-\bC_\bv\big|^q\,\dd\bx
\ge
C_M\Big(\|\bu-\bv\|_{W^{1,p}(\Omega)}\Big)^p
\]
\n for all $\bu\in W^{1,2q}(\Omega;\R^n)$ that satisfy
$\det\grad\bu>0~a.e.$ in $\Om$ and $\bu=\bv$ on $\sD$.
\end{proposition}

\begin{remark}\label{rem:CM-1}  We note that $\bC_{\bu}=\bC_{\bv}$
does not necessarily imply that $\bu=\bv$ without further assumptions.
See, e.g., Ciarlet \& Mardare~\cite[p.~425]{CM04-1},
who attribute their counterexample to H.~Le~Dret (and a referee),
or Lorent~\cite[p.~659]{Lo13}.
\end{remark}


\subsection{Bounded Mean Oscillation}\label{sec:BMO}

We define the $\BMO$-seminorm\footnote{See, e.g., \cite[\S3.1]{Gr09}
for properties of $\BMO$.  Note that $\tsn{c}_{\BMO(U)} =0$ for any
constant $c$ and, otherwise, $\tsn{\! \cdot\!}_{\BMO(U)}$ obeys
the properties of a norm. Moreover, $\BMO(U)$ is complete with respect
to this seminorm.} of $\psi\in L^1_{\loc}(U)$ by
\be\label{eqn:BMO-V-norm}
\tsn{\psi}_{\BMO(U)}:=  \sup_{Q\Ksubset U}
 \dashint_Q |\psi(\bx)-\langle\psi\rangle_Q|\,\dd\bx,
\ee
\n where the supremum is to be taken over all
nonempty, bounded (open) $n$-dimensional
hypercubes\footnote{We shall henceforth refer to
$Q$ as a \emph{cube}, rather than a hypercube or square.}
 $Q$  \emph{with faces parallel to the coordinate hyperplanes}.  Here
\[
\langle\psi\rangle_U
:=\dashint_U \psi(\bx)\,\dd\bx
:= \frac1{|U|}\int_U \psi(\bx)\,\dd\bx
\]
\n \emph{denotes the average value of} $\psi$ and  $|U|$ denotes the
$n$-dimensional Lebesgue measure of any
bounded domain $U\subset\R^n$.
The space $\BMO(U)$ (Bounded Mean Oscillation) is defined by
\be\label{eqn:BMO-V}
\BMO(U)
:=
\{\psi\in L_{\loc}^1(U): \tsn{\psi}_{\BMO(U)} < \infty\}.
\ee
\n Note that one consequence of
\eqref{eqn:BMO-V-norm}--\eqref{eqn:BMO-V} is that
$L^\infty(U)\subset \BMO(U)$ with
\be\label{eqn:BMO-L-infty}
\tsn{\psi}_{\BMO(U)} \le 2\|\psi\|_{\infty,U}\
\text{ for all $\psi\in L^\infty(U)$.}
\ee


\n We note for future reference that if $U=\Om$, a Lipschitz domain,
then a result of P.~W.~Jones~\cite{Jo82} implies, in particular, that
\[
\BMO(\Om)\subset L^1(\Om).
\]
\n It follows that
\be\label{eqn:BMO-Om-norm}
\|\psi\|_{\BMO(\Om)}=\tsn{\psi}_{\BMO(\Om)} + | \langle\psi\rangle_{\Om}|
\ee
\n is a \emph{norm} on $\BMO(\Om)$.

\begin{remark}\label{rem:other-BMO}  1.~The standard example of a
function $\phi\in\BMO(\R^n)$ that is not bounded is
$\phi(\bx)=\ln|\bx|$.  2.~There are a number of other
\emph{equivalent seminorms} on $\BMO$.  The most
ubiquitous
involves the replacement of cubes $Q$ in \eqref{eqn:BMO-V-norm} by
open balls $B\Ksubset U$.   Another possibility is the use of balls
that get smaller as they
approach the boundary (see Brezis \& Nirenberg~\cite{BN96} who attribute
such results to P.~W.~Jones~\cite{Jo82}), i.e., the requirement that
there is a fixed $k\in (0,1)$ such that each ball,
$B=B_r(\bx)\subset\subset U$ of radius $r>0$ and centered at $\bx$,
satisfies
\[
r\le k\dist(\bx,\partial U).
\]
\n Another useful equivalent seminorm is
\[
\tsn{\psi}_*:=\sup_{Q\Ksubset U}
\dashint_Q \dashint_Q |\psi(\bz)-\psi(\bx)|\,\dd\bz\,\dd\bx;
\]
\n in particular (see, e.g., \cite[p.~6]{BN95})
\[
\tsn{\psi}_{\BMO(U)}\le\tsn{\psi}_*\le2\tsn{\psi}_{\BMO(U)}.
\]
\n The monotone convergence theorem can then be used show that
\[
\sup_{Q\subset U}
\dashint_Q \dashint_Q |\psi(\bz)-\psi(\bx)|\,\dd\bz\,\dd\bx
\]
\n is also an equivalent seminorm on $\BMO(U)$; it then follows that the
seminorm
\[
  \sup_{Q\subset U}
 \dashint_Q |\psi(\bx)-\langle\psi\rangle_Q|\,\dd\bx,
\]
\n which is used in \cite{DRS10,SS19},
is also equivalent to \eqref{eqn:BMO-V-norm}.
\end{remark}


\subsection{Further Properties of
\texorpdfstring{$\BMO$}{BMO}}\label{sec:maximal}

One of the main properties of $\BMO$ that we shall use is contained
in the following result.  Although the proof can be found in
\cite{SS19}, the significant analysis it is based upon is due to
Fefferman \& Stein~\cite{FS72}, Iwaniec~\cite{Iw82}, and Diening,
R$\overset{_\circ}{\mathrm{u}}$\v zi\v cka, \& Schumacher~\cite{DRS10}.


\begin{proposition}\label{thm:main-2}  Let $\Om\subset\R^n$ be a
Lipschitz\footnote{This result, as stated, is valid for a larger
class of domains:  Uniform domains.  (Since $\BMO\subset L^1$
for such domains.  See P.~W.~Jones~\cite{Jo82},
Gehring \& Osgood~\cite{GO79}, and e.g., \cite{Ge87}.)  A slightly
modified version of this result is valid for John domains.
 See \cite{SS19} and the references therein.}
domain.  Then, for all $q\in[1,\infty)$,
\[
\BMO(\Om)\subset L^q(\Om)
\]
\n with continuous injection, i.e., there is a constant
$J_1=J_1(q,\Om)>0$ such that,
for every $\psi\in\BMO(\Om)$,
\be\label{eqn:BMO-in-Lq}
\bigg(\,\dashint_\Om |\psi|^q\,\dd\bx\bigg)^{\!\!1/q}\!
\le
J_1\|\psi\|_{\BMO(\Om)}.
\ee
\n Moreover, if  $1\le p<q<\infty$,
then there exists a
constant $J_2=J_2(p,q,\Om)>0$ such that every
$\psi\in\BMO(\Om)$ satisfies
\be\label{eqn:RH}
||\psi||_{q,\Om}
\le
J_2\Big(||\psi||_{\BMO(\Om)}\Big)^{1-p/q}
\Big(||\psi||_{p,\Om}\Big)^{p/q}.
\ee
\n In addition, the constants $J_i$ are scale invariant, that is,
$J_i(\lambda U +\ba)=J_i(U)$ for every $\lambda>0$ and $\ba\in\R^n$.
Here (see (\ref{eqn:BMO-V-norm}))
$\|\cdot\|_{\BMO(\Om)}$ is given by (\ref{eqn:BMO-Om-norm}).
\end{proposition}


\begin{remark}\label{rem:subset} Proposition~\ref{thm:main-2}
together with
\eqref{eqn:BMO-L-infty} shows that, for every $p\in[1,\infty)$,
\[
L^\infty(\Om) \subset \BMO(\Om)\subset L^p(\Om).
\]
Thus, $\BMO$ is a space that is ``between'' $L^\infty$ and all of the
other $L^p$-spaces.  However, researchers in Harmonic Analysis make use
of $\BMO$ as a replacement for $L^\infty$.
See, e.g., \cite[\S4.5]{Stein-1993}.
\end{remark}


According to L.~Nirenberg (\cite[pp.~707--709]{Jo85b}), the idea of
considering functions whose mean oscillation is bounded was conceived by
Fritz John. John's motivation appears to have been the analysis of
problems in Nonlinear Elasticity, where John had noticed that mappings
with small nonlinear strain (see \eqref{eqn:GSVST}) correspond to
deformation gradients the are small in $\BMO$ (see \cite{Jo61} or, e.g.,
\cite[Proposition~4.3 and Lemma~5.6]{SS19}).

The final result of this section follows from
Proposition~\ref{thm:main-2}. However, since the result
is a direct consequence of the scale invariance of the constant in
the same result for cubes this
result also follows from the original proof of
John \& Nirenberg~\cite{JN61}.

\begin{corollary}\label{cor:JN} Fix $n\ge2$.  Then, for every
$q\in(1,\infty)$, there exists a constant $\sN=\sN(n,q)$ such that,
for every bounded domain $V\subset\R^n$,
\be\label{eqn:JN2}
\tsn{\phi}_{\BMO(V)}
\le
\sup_{Q\Ksubset V}\bigg(\,
\dashint_Q |\phi(\bx)-\langle\phi\rangle_Q|^q\,\dd\bx\bigg)^{\!\!1/q}
\le
\sN(n,q)\tsn{\phi}_{\BMO(V)}.
\ee
\n for all $\phi\in\BMO(V)$.
\end{corollary}


\begin{remark}\label{rem:JN} Corollary~\ref{cor:JN} shows that
\[
\tsn{\phi}_{\BMO_q(V)}:=\sup_{Q\Ksubset V}
\bigg(\,
\dashint_Q |\phi(\bx)-\langle\phi\rangle_Q|^q\,\dd\bx\bigg)^{\!\!1/q}
\]
\n is an equivalent seminorm on $\BMO(V)$. This result was first
established by John \& Nirenberg~\cite{JN61} when $V=Q$, a cube; it is
there a consequence of what is now referred to as the John-Nirenberg
inequality, that is, the exponential decay of the distribution
function of  $|\phi-\langle\phi\rangle_Q|$ for cubes.
Inequality \eqref{eqn:JN2}  is also well-known when $V$ and $Q$ are
replaced by balls $B\Ksubset\widehat B$; see, e.g.,
Stein~\cite[pp.~144--146]{Stein-1993}.  Stein also shows that, for balls
$\widehat B$, the constant $\sN$ satisfies
$\sN(q,\widehat B)\le q\widehat \sN(\widehat B)$; the exponential decay of
$|\phi-\langle\phi\rangle_{\widehat B}|$ for balls then
follows from \eqref{eqn:JN2} and this growth estimate.
\end{remark}


\begin{proof}[Proof of Corollary~\ref{cor:JN}]  Let $q\in(1,\infty)$ and
suppose that  $V\subset\R^n$ is a bounded domain. Fix a cube
$Q \Ksubset V$.
Then, in view of \eqref{eqn:BMO-in-Lq} in Theorem~\ref{thm:main-2}
(with $\Om=Q$) and \eqref{eqn:BMO-Om-norm}, there exists a scale
invariant constant $J_1(q,Q)$
such that, for all $\psi\in\BMO({Q})$,
\be\label{eqn:BN-2}
J_1^{\mi1}\bigg(\,\dashint_{Q } |\psi|^q\,\dd\bx\bigg)^{\!\!1/q}\!
\le
\tsn{\psi}_{\BMO({Q })}
+\Big|\dashint_{Q } \psi\,\dd\bx\Big|.
\ee


Now, suppose that $\phi\in\BMO(V)$; then
$\phi\in\BMO(Q)$.  Define
$\psi:=\phi-\langle\phi\rangle_{Q }$.  Thus,
$\psi\in\BMO(Q)$,
$\langle\psi\rangle_{Q }=0$, and hence \eqref{eqn:BN-2} yields
\be\label{eqn:BN-3}
J_1^{\mi1}
\bigg(\,\dashint_{Q} |\phi-\langle\phi\rangle_{Q}|^q
\,\dd\bx\bigg)^{\!\!1/q}\!
\le
\tsn{\phi-\langle\phi\rangle_{Q }}_{\BMO({Q })}
=
\tsn{\phi}_{\BMO({Q})}.
\ee
\n Note that
\[
\tsn{\phi}_{\BMO({Q})}
:=
\sup_{\widehat Q\Ksubset Q}
\dashint_{\widehat Q}
|\phi(\bx)-\langle\phi\rangle_{\widehat Q}|\,\dd\bx
\le
\sup_{\widehat Q\Ksubset V}
\dashint_{\widehat Q}
|\phi(\bx)-\langle\phi\rangle_{\widehat Q}|\,\dd\bx
=:
\tsn{\phi}_{\BMO(V)},
\]
\n which together with \eqref{eqn:BN-3} and H\"older's inequality
yields
\be\label{eqn:BN-4}
\dashint_{Q} |\phi-\langle\phi\rangle_{Q}|
\,\dd\bx
\le
\bigg(\,\dashint_{Q} |\phi-\langle\phi\rangle_{Q }|^q
\,\dd\bx\bigg)^{\!\!1/q}\!
\le
J_1\tsn{\phi}_{\BMO(V)}.
\ee
\n The desired result, \eqref{eqn:JN2}, now follows after
taking the supremum of \eqref{eqn:BN-4} over all cubes
$Q\Ksubset V$ and noting that $\sN(n,q):=J_1(n,q,Q)$ is scale
invariant and hence independent of the cube.
\end{proof}


\section{Korn's Inequality}\label{sec:Korn}

In this section we shall obtain a version of Korn's inequality that
involves the $\BMO$-seminorm of both the gradient of a function and the
symmetric part of its gradient.  Our result is a simple consequence
of the following result of Diening,
R$\overset{_\circ}{\mathrm{u}}$\v zi\v cka, \& Schumacher.

\begin{proposition}\label{prop:Korn2-alt}
\emph{(\cite[Theorem~5.17]{DRS10})}
Let $\Om\subset\R^n$, $n\ge2$,
be a bounded Lipschitz\footnote{In \cite{DRS10} this
result is established for John domains.} domain.
Suppose that $q\in(1,\infty)$.
Then there exists a  scale invariant constant $K=K(q,\Om)>0$
such that, for all $\bu\in W^{1,q}(\Om;\R^n)$,
\be\label{eqn:Korn2}
\dashint_\Om
\left|\grad\bu-\langle\grad\bu\rangle_\Om\right|^q\,\dd\bx
\le
K\dashint_\Om
\left|\grads\bu-\langle\grads\bu\rangle_\Om\right|^q\,\dd\bx,
\ee
\n where $\grads\bu$ denotes the symmetric part of the gradient of $\bu$,
that is,
\[
\grads\bu:= \tfrac12\left[\grad\bu+(\grad\bu)^\rmT\right].
\]
\end{proposition}


\begin{remark} The scale invariance of
$K$ is clear since the average value of any function is scale invariant.
\end{remark}


\subsection{Korn's Inequality in \texorpdfstring{$\BMO$}{BMO}}

\begin{theorem} \label{prop:BMO-Korn} Fix $n\ge2$.
Then there exists a constant $\sK=\sK(n)>0$
such that, for any bounded domain $U\subset\R^n$,
\be\label{eqn:BMO-Korn}
\tsn{\grad\bu}_{\BMO(U)}\le \sK\tsn{\grads\bu}_{\BMO(U)},
\ee
\n for every $\bu\in W_{\loc}^{1,1}(U;\R^n)$ with
$\grad\bu\in \BMO(U)$.
\end{theorem}


\begin{remark} Note that the above inequality is valid for
\textbf{all bounded
domains}, unlike the standard Korn inequalities, which are valid only for
John domains (see, e.g., \cite{JK17}).
\end{remark}


\begin{proof}[Proof of Theorem~\ref{prop:BMO-Korn}]  Fix a bounded domain
$U\subset\R^n$. Let
$\bu\in W_{\loc}^{1,1}(U;\R^n)$ with
$\grad\bu\in \BMO(U)$ and suppose that  $Q\Ksubset U$ is a cube.
The definition of $\BMO(U)$, \eqref{eqn:BMO-V}, yields
$\grad\bu\in\BMO(Q)\cap L^1(Q)$.  Inequality~\eqref{eqn:BMO-in-Lq} in
Proposition~\ref{thm:main-2} then implies that
$\grad\bu\in L^q(Q)$ for every $q\in[1,\infty)$.


Next, by Korn's inequality, Proposition~\ref{prop:Korn2-alt}, there
exists a constant $K=K(2,Q)$, which is independent of $\bu$, such
that
\be\label{eqn:Korn2-cube-2}
\bigg(\,\dashint_Q
\left|\grad\bu-\langle\grad\bu\rangle_Q\right|\,\dd\bx\bigg)^{\!2}
\le
\dashint_Q
\left|\grad\bu-\langle\grad\bu\rangle_Q\right|^2\dd\bx
\le
K\dashint_Q
\left|\grads\bu-\langle\grads\bu\rangle_Q\right|^2\dd\bx,
\ee
\n where the first inequality in \eqref{eqn:Korn2-cube-2}
follows from H\"older's inequality.
If we now take the supremum of \eqref{eqn:Korn2-cube-2}
over all cubes $Q\Ksubset U$ and
make use of the scale invariance of $K$ and the definition of the
$\BMO$-seminorm, \eqref{eqn:BMO-V-norm}, we find that
\be\label{eqn:last}
\Big(\tsn{\grad\bu}_{\BMO(U)}\Big)^{2}
\le
K\sup_{Q\Ksubset U}\dashint_Q
\left|\grads\bu-\langle\grads\bu\rangle_Q\right|^2\dd\bx.
\ee
\n The desired result, \eqref{eqn:BMO-Korn}, now follows from
\eqref{eqn:last}, Corollary~\ref{cor:JN},
and the fact that the
constant $\sN=\sN(n,q)$ in \eqref{eqn:JN2} is independent of the domain.
\end{proof}


\section{Nonlinear Elasticity}\label{sec:NLE}

In the remainder of this manuscript we shall focus on the minimization
problem that arises when one considers the theory of Nonlinear
(Finite) Hyperelasticity.


\subsection{The Constitutive Relation}\label{sec:CR}

We consider a \emph{body} that for convenience we identify with the
closure of a \emph{bounded Lipschitz domain} $\Omega\subset\R^n$,
$n=2$ or $n=3$, which it
occupies in a fixed reference configuration.  A \emph{deformation}
of $\overline{\Omega}$ is a mapping that lies in the space
\[
\Def:=\{\bu\in W^{1,2}(\Omega;\R^n): \det\grad\bu>0\ a.e.\},
\]
\n where $\det\bF$ denotes the determinant of $\bF\in\Mn$.
 We assume that the body is composed of a hyperelastic material with
\emph{stored-energy density}\footnote{Our formulation implicitly
assumes that the response of the material is
invariant under a change in observer.  See, e.g., \cite[\S48]{GFA10}.}
$\sigma:\overline{\Omega}\times\Psymn\to[0,\infty)$.  The quantity
$\sigma(\bx,\bC_\bu(\bx))$ gives the elastic energy stored
at almost every point $\bx\in\Omega$ of a deformation $\bu\in\Def$.
Here, and in the sequel,  $\bC_\bu=[\grad\bu]^\rmT\grad\bu$, the
right Cauchy-Green strain tensor  (cf.~\eqref{eqn:RCGST}).


\begin{hypothesis}\label{def:W}  We assume that $\sigma$ satisfies
the following:
\begin{enumerate}[topsep=-2pt]
\item $\bC\mapsto \sigma(\bx,\bC) \in C^3(\Psymn)$,
for $a.e.~\bx\in\Omega$;
\item $(\bx,\bC)\mapsto\DD^k \sigma(\bx,\bC)$, $k=0,1,2,3$, are each
(Lebesgue) measurable on their common domain
$\Omega\times\Psymn$; and
\item  $(\bx,\bC)\mapsto\DD^k \sigma(\bx,\bC)$, $k=0,1,2,3$,
are each bounded on $\overline{\Omega}\times K$ for every
nonempty compact $K\subset \Psymn$.  Here
\end{enumerate}\vspace{-.2cm}
\[
\DD^0 \sigma(\bx,\bC):=\sigma(\bx,\bC), \qquad
\DD^k \sigma(\bx,\bC)
:=
\frac{\partial^k}{\partial\bC^k}\sigma(\bx,\bC)
\]
\n denotes $k$-th derivative of the function
$\bC\mapsto \sigma(\cdot,\bC)$.
We note, in particular, that
\[
\DD\sigma:\overline{\Omega}\times\Psymn\to\Symn, \qquad
\DD^2\sigma:\overline{\Omega}\times\Psymn\to\Lin(\Symn;\Symn),
\]
\n where $\Lin(\mathcal{U};\mathcal{V})$ denotes the
set of linear maps from
the vector space $\mathcal{U}$ to the vector space $\mathcal{V}$.  Thus,
for every $\bC\in\Psymn$, $\bE\in\Symn$, and almost every $\bx\in\Om$,
\[
\DD\sigma(\bx,\bC)\in \Symn, \quad
\DD^2\sigma(\bx,\bC)[\bE]\in \Symn\!.
\]
\end{hypothesis}


The \emph{second Piola-Kirchhoff stress tensor} $\bK$ is defined to be
twice the derivative of $\sigma$, i.e.,
\be\label{eqn:sPKS}
\bK(\bx,\bC):= 2\frac{\partial}{\partial\bC}\sigma(\bx,\bC)
=2\DD\sigma(\bx,\bC),
\qquad \bK :\overline{\Omega}\times\Psymn\to\Symn,
\ee
\n while the \emph{first Piola-Kirchhoff stress tensor} $\bS$
is given by
\be\label{eqn:fPKS}
\bS(\bx,\bF) := \bF \bK(\bx,\bF^\rmT\bF),
\qquad \bS :\overline{\Omega}\times\Mnp\to\Mn\!,
\ee
\n where $\Mnp$ denotes the set of $n$ by $n$ matrices with positive
determinant.  Although the tensor $\bK$ is the derivative of the stored
energy, it is the tensor $\bS$ that is most convenient to use in the
equilibrium (Euler-Lagrange) equations
(see \eqref{eqn:EE}--\eqref{eqn:SEE}).
For any injective deformation $\bu\in\Def\cap\, C^1(\Omc;\R^n)$, the
\emph{Cauchy stress tensor} $\bT=\bT(\by)\in\Symn$
is given by
\be\label{eqn:CS}
\bT(\by)
:=\bF\bK(\bx,\bF^\rmT\bF)\bF^\rmT(\det\bF)^{\mi1},
\quad \text{$\bF:=\grad\bu(\bx)$, $\by=\bu(\bx)$.}
\ee
\n The eigenvalues of $\bT(\by)$ are called \emph{the principal stresses}
at $\by\in\bu(\Om)$.
  The  \emph{elasticity tensor} $\CC$ is  defined to be four times
the second derivative of
$\bC\mapsto \sigma(\bx,\bC)$, that is,
\be\label{eqn:ET-A}
\CC(\bx,\bC):=4\frac{\partial^2}{\partial\bC^2} \sigma(\bx,\bC)
=4\DD^2\sigma(\bx,\bC).
\ee
\n In view of the symmetry of the second gradient,
\[
\bB:\CC(\bx,\bC)[\bE]=\bE:\CC(\bx,\bC)[\bB]
\]
\n for all $\bC\in\Psymn$ and all $\bB,\bE\in\Symn$.


\begin{definition}  The elasticity tensor is said to be
\emph{uniformly positive definite} at a deformation
$\bu\in\Def$
provided that there exists
a constant $c>0$ such that, for every $\bE\in\Symn$
and $a.e.~\bx\in\Omega$,
\[
\bE:\CC\big(\bx,\bC_\bu(\bx)\big)[\bE]\ge c |\bE|^2,
\]
\n where $\bC_\bu=(\grad\bu)^\rmT\grad\bu$.  The reference
configuration is said to be \emph{stress free} provided that,
\[
\bK(\bx,\bI) =\mathbf{0}\ \text{ for $a.e.~\bx\in\Omega$},
\]
\n where $\bI\in\Mn$ denotes the identity matrix.
\end{definition}

\begin{remark}  Let $n=3$ and suppose that
$\bu\in\Def\cap\, C^1(\Omega;\R^n)$ is injective.
Further, let $\Gamma\subset\bu(\Om)$
be a smooth, oriented surface with continuous outward unit normal
field $\by\mapsto\bm(\by)$, $\by\in\Gamma$.
If $\bx\mapsto\bK(\bx,\bC)$ is continuous on $\Om$, then,
for any $\by=\bu(\bx)$ with $\by\in\Gamma$,
\[
\bT(\by)\bm(\by),
\]
\n gives the force, per unit (deformed) area, exerted across $\Gamma$
upon the material on the negative side of $\Gamma$ by the material on
the positive side $\Gamma$
(see, e.g., \cite[p.~97]{Gu81} or \cite[\S19]{GFA10}).
\end{remark}

\begin{remark} One can alternatively assume that the
stored-energy function depends
on the deformation gradient $\grad\bu$. In this case one postulates a
(frame-indifferent) function $W:\overline{\Omega}\times\Mnp\to[0,\infty)$,
which will satisfy, for $a.e.~\bx\in\Om$,
\be\label{eqn:W=sigma}
W(\bx,\bF) = \sigma(\bx,\bF^\rmT\bF)\ \text{ for every $\bF\in\Mnp$}.
\ee
\n In this formulation one usually defines the elasticity tensor as
the second derivative of $W$ with respect to $\bF$, that is,
\be\label{eqn:ET-alt}
\A(\bx,\bF):=\frac{\partial^2}{\partial\bF^2} W(\bx,\bF).
\ee
\n  If we now twice differentiate \eqref{eqn:W=sigma}
(see, e.g., \cite[Lemma~5.4]{SS19}),
we conclude, with the
aid of \eqref{eqn:sPKS}, \eqref{eqn:ET-A},
\eqref{eqn:ET-alt}, and the symmetry of the second derivative,
 that
\be\label{eqn:A=CplusK}
\begin{aligned}
\bH:\A(\bx,\bF)[\bH]
&=\tfrac12(\bH^\rmT\bF+\bF^\rmT\bH):
\CC(\bx,\bF^\rmT\bF)\big[\tfrac12(\bH^\rmT\bF+\bF^\rmT\bH)\big]\\
&\ \ +\bK(\bx,\bF^\rmT\bF):(\bH^\rmT\bH),
\end{aligned}
\ee
\n for all $\bF\in \Mnp$ and $\bH\in\Mn$.  In particular, when the reference
configuration is stress free, it follows that
\[
\bH:\A(\bx,\bI)[\bH]
=\tfrac12(\bH^\rmT + \bH):
\CC(\bx,\bI)\big[\tfrac12(\bH^\rmT + \bH)\big].
\]
\n Thus, both $\CC(\bx,\bI)$ and $\A(\bx,\bI)$ correspond to the classical
elasticity tensor used in the linear theory (see, e.g., \cite{Gu72}).
\end{remark}


\subsection{Equilibrium Solutions and Energy
Minimizers}\label{sec:ES-EM-NLE}

 We assume the body is subject to
dead loads.  We take
\[
\partial\Omega = \overline{\sD} \cup \overline{\sS}\quad
\text{with $\sD$ and $\sS$ relatively open and }
\sD\cap\sS=\varnothing.
\]
\n In addition, we shall suppose that $\sD\ne\varnothing$.
We assume that a Lipschitz-continuous function
$\bd:\sD\to\R^n$
is prescribed;  $\bd$ will give the deformation of $\sD$.
If $\sS \ne\varnothing$ we assume that a function
$\bs\in L^2(\sS;\R^n)$
is prescribed; for $\sH^{n-1}$-$a.e.~\bx\in\sS$,
$\bs(\bx)$ will give the
surface force (per unit area when $n=3$)
exerted on the body at the point $\bx$ by its environment.
Here $\sH^k$ denotes
$k$-dimensional Hausdorff measure\footnote{Thus,
when $\sS\subset\R^3$ is a
smooth surface, $\sH^2(\sS)$ gives the area of $\sS$.}
(see, e.g., \cite[Chapter~2]{EG92}).
Finally, we suppose that a
function $\bb\in L^2(\Omega;\R^n)$ is prescribed;
for $a.e.~\bx\in\Omega$, $\bb(\bx)$ will give the
body force (per unit volume when $n=3$)
exerted on the body at the point $\bx$ by its environment.
  The set of
\emph{Admissible Deformations} will be denoted by
\[
\AD:=\{\bu\in\Def \cap\, W^{1,\infty}(\Omega;\R^n):
\bu=\bd \text{ on $\sD$}\}.
\]
\n The  \emph{total energy} of an admissible deformation
$\bu\in\AD$ is defined to be
\be\label{eqn:TE}
\E(\bu):=\int_\Omega \Big[\sigma\big(\bx,\bC_\bu(\bx)\big)
-\bb(\bx)\cdot\bu(\bx)\Big]\dd\bx
-\int_{\sS} \bs(\bx)\cdot\bu(\bx)\,\dd\sH^{n-1}_\bx
\ee
\n with $\bC_\bu:=(\grad\bu)^\rmT\grad\bu$.
The first variation of $\E$ is
given by
\[
\begin{aligned}
\delta\E(\bu)[\bw]=&\int_\Omega \DD\sigma\big(\bx,\bC_\bu(\bx)\big):
\Big(\big[\grad\bu(\bx)\big]^\rmT\grad\bw(\bx)
+\big[\grad\bw(\bx)\big]^\rmT\grad\bu(\bx)\Big)\dd\bx\\
&-\int_\Omega\bb(\bx)\cdot\bw(\bx)\,\dd\bx
-\int_{\sS} \bs(\bx)\cdot\bw(\bx)\,\dd\sH^{n-1}_\bx,
\end{aligned}
\]
\n for all \emph{variations} $\bw\in\Var$, where
\[
\Var:=\{\bw\in W^{1,2}(\Omega;\R^n):
\bw=\mathbf{0} \text{ on $\sD$}\}.
\]
\n The second variation of $\E$ is then given by
(see \eqref{eqn:sPKS}, \eqref{eqn:ET-A}, and \eqref{eqn:A=CplusK})
\be\label{eqn:SV}
\begin{aligned}
\delta^2\E(\bu)&[\bw,\bw]
= \int_\Omega \bK\big(\bx,\bC_\bu(\bx)\big):
\big[(\grad\bw)^\rmT\grad\bw\big]\dd\bx\\
&+\tfrac14\int_\Omega \big[(\grad\bu)^\rmT\grad\bw+(\grad\bw)^\rmT\grad\bu\big]:
\CC\big(\bx,\bC_\bu(\bx)\big)
\big[(\grad\bu)^\rmT\grad\bw+(\grad\bw)^\rmT\grad\bu\big]\,\dd\bx.
\end{aligned}
\ee

\begin{remark}\label{rem:PSV} It is clear from \eqref{eqn:SV} that the
positivity of the second variation, i.e., $\delta^2\E(\bu)\ge0$, is
not a consequence of the positivity of the elasticity tensor $\CC$
alone.  However, the
second variation is positive whenever both $\CC$  and
$\bK$ are positive definite (see Lemma~\ref{lem:L-posdef}).
\end{remark}

We shall assume that we are given a deformation
$\bu\e\in\AD$ that is a weak
solution of the \emph{Equilibrium Equations} corresponding
to \eqref{eqn:TE}, i.e., $\delta\E(\bu)=0$ or, equivalently,
\be\label{eqn:EE}
0=\int_\Omega
\Big[\bS\big(\bx,\grad\bu\e(\bx)\big)
:\grad\bw(\bx)
-\bb(\bx)\cdot\bw(\bx)\Big]\dd\bx
-\int_\sS \bs(\bx)\cdot\bw(\bx)\,\dd \sH^{n-1}_\bx
\ee
\n for all $\bw\in\Var$, where
$\bS$ is given by  \eqref{eqn:sPKS}--\eqref{eqn:fPKS}.
%
If $\sD=\partial\Om$ we shall call $\bu\e$ a solution of the
\emph{(pure) displacement problem}.  Otherwise, we shall refer to
such a  $\bu\e$ as a solution of the
\emph{(genuine) mixed problem}.
If in addition $\sigma\in C^2(\Omc\times\Psymn)$ and
$\bu\e\in C^2(\Omega;\R^n)\cap C^1(\overline{\Omega};\R^n)$, then
$\bu\e$ will be a \emph{classical solution of the equations of
equilibrium} (see, e.g., \cite[\S2.6]{Ci1988}, \cite[\S27]{Gu81},
or \cite[\S49]{GFA10}),
i.e., $\bu\e$ will satisfy
\be\label{eqn:SEE}
\begin{gathered}
\Div \bS(\grad\bu\e) +\bb= \mathbf{0}\
\text{ in $\Omega$,}\\
 \bS(\grad\bu\e)\bn = \bs\
\text{ on $\sS$,}\qquad
\bu\e=\bd\
\text{ on $\sD$,}
\end{gathered}
\ee
\n where $\bn(\bx)$ denotes the outward unit normal to $\Omega$ at
$\sH^{n-1}$-$a.e.~\bx\in\sS$ and $\Div \bM\in\R^n$ is given by
$(\Div \bM)_i=\sum_j \frac{\partial}{\partial \bx_j} M_{ij}$.


We are interested in the local minimality (in an appropriate
topology) of solutions of  \eqref{eqn:EE}.
For future use we note that, for every $\bu,\bv\in\AD$,
\eqref{eqn:TE} gives us
\[
\E(\bv)-\E(\bu)
= \int_\Omega
\big[\sigma\big(\bC_\bv\big)-\sigma\big(\bC_\bu\big)
-\bb\cdot\bw\big]\dd\bx
-
\int_\sS \bs\cdot\bw\,\dd\sH^{n-1}_\bx,
\]
\n where $\bw:=\bv-\bu\in W^{1,\infty}(\Omega;\R^N) \cap \Var$.
It follows that, when $\bu\e\in\AD$ is a solution of
the equilibrium equations, \eqref{eqn:EE},
we have the identity, for every $\bv\in\AD$,
\be\label{eqn:Identity-ES}
\E(\bv)-\E(\bu\e)
= \int_\Omega
\Big[\sigma\big(\bx,\bC_\bv(\bx)\big)-
\sigma\big(\bx,\bC\e(\bx)\big)
-\bS\big(\bx,\grad\bu\e(\bx)\big)
:\grad\bw(\bx)\Big]\dd\bx,
\ee
\n where $\bC\e:=\bC_{\bu\e}=(\grad\bu\e)^\rmT\grad\bu\e$,
$\bC_\bv:=(\grad\bv)^\rmT\grad\bv$, and $\bw:=\bv-\bu$.


\subsection{Multiaxial Tension}

In the sequel we shall assume that the second Piola-Kirchhoff stress
tensor $\bK$ is positive semidefinite at a given deformation
$\bv\in\AD$, that is, for $a.e.~\bx\in\Om$,
\be\label{eqn:stress-tensile}
\ba\cdot\bK\big(\bx,\bC_\bv(\bx)\big)\ba\ge 0\
\text{ for every $\ba\in\R^n$.}
\ee
\n In view of \eqref{eqn:CS} and the positivity of the Jacobian
$\det\grad\bv$, inequality \eqref{eqn:stress-tensile} is
essentially the same as the
assumption that the Cauchy stress tensor $\bT$ is positive semidefinite.
Thus, \eqref{eqn:stress-tensile} is the assumption that \emph{the
principal stresses in the deformed material are all tensile.}

The next result yields a simple consequence of
\eqref{eqn:stress-tensile}
that we shall use.  We sketch a proof
for the convenience of the reader.


\begin{lemma}\label{lem:L-posdef} Let $\bL\in L^p(\Om;\Symn)$ for some
$p\in[1,\infty]$.
Suppose that $\bL(\bx)$ is
positive semidefinite at almost every $\bx\in\Om$.  Then
\be\label{eqn:STEL}
I(\bw)=\int_\Om
\big(\big[\grad\bw(\bx)\big]^\rmT\grad\bw(\bx)\big):\bL(\bx)
\,\dd\bx\ge0
\ee
\n for all $\bw\in W^{1,q}(\Om;\R^n)$, where
\be\label{eqn:p-q}
\frac1p+\frac2{q}=1.
\ee
\n Conversely, suppose that
$\bL\in C(\Om;\Symn)$ satisfies
\eqref{eqn:STEL} for all  $\bw\in W_0^{1,2}(\Om;\R^n)$.
Then $\bL(\bx)$ is
positive semidefinite at every $\bx\in\Om$.
\end{lemma}


\begin{proof}  Fix $p\in[1,\infty]$. Let $\bL\in L^p(\Om;\Symn)$
with $\bL(\bx)$  positive semidefinite at $a.e.~\bx\in\Om$.
Then, by the spectral theorem, at $a.e.~\bx\in\Om$
there exists an orthonormal basis
$\bee_k(\bx)$ and scalars $\alpha_k(\bx)$, $k=1,2,\ldots,n$, with
$\alpha_k\ge0~a.e.$, such that
\be\label{eqn:spectral}
\bL(\bx)=\sum_{k=1}^n \alpha_k(\bx)\bee_k(\bx)\otimes\bee_k(\bx),
\ee
\n where $\ba\otimes\bb\in\Mn$ is defined by
$[\ba\otimes\bb]\bc=(\bb\cdot\bc)\ba$ for every $\bc\in\R^n$.

Let $\bw\in W^{1,q}(\Om;\R^n)$, where $q$ satisfies \eqref{eqn:p-q}.
Then \eqref{eqn:spectral} yields, with the aid of the inequalities
$\alpha_k\ge0~a.e.$,
\be\label{eqn:to-integrate}
\big(\big[\grad\bw\big]^\rmT\grad\bw\big):\bL
=
\sum_{k=1}^n \alpha_k\big|(\grad\bw)\bee_k\big|^2\ge 0~a.e.
\ee
\n Since $[\grad\bw]^\rmT\grad\bw\in L^{q/2}(\Om;\Symn)$ and
$\bL\in L^p(\Om;\Symn)$,  equation \eqref{eqn:p-q}
implies that \eqref{eqn:to-integrate} is integrable.
Thus, we may integrate
\eqref{eqn:to-integrate} over $\Om$ to arrive at
\eqref{eqn:STEL}.


Conversely, suppose that $\bL\in C(\Om;\Symn)$ satisfies \eqref{eqn:STEL}
for all  $\bw\in W_0^{1,2}(\Om;\R^n)$.
Note that \eqref{eqn:STEL} is the condition that $I$ assumes its infimum
at $\bw=\mathbf{0}$.  A standard result (see, e.g., \cite{Me65} or
\cite[Theorem~2.2(i)]{BM84}) is that $I$ is then
quasiconvex at $\bw=\mathbf{0}$; thus, for every $\bx\oo\in\Om$ and
$\widehat\bw\in W_0^{1,2}(B;\R^n)$,
\be\label{eqn:I-hat}
\widehat I (\widehat\bw)=
\int_B
\big(\big[\grad\widehat\bw(\bz)\big]^\rmT\grad\widehat\bw(\bz)\big):\bL(\bx\oo)
\,\dd\bz\ge0,
\ee
\n where $B\subset\R^n$ denotes the unit ball centered at $\mathbf{0}$.
In particular, fix $\bee\in\R^n$ with $|\bee|=1$ and let
$\widehat\bw=\phi\;\!\bee$, where $\phi\in W_0^{1,2}(B)$.
We then find, with the aid of the spectral theorem
(see \eqref{eqn:spectral}), that \eqref{eqn:I-hat} reduces to
\[
0\le\widehat I (\phi\;\!\bee)=
\int_B
\grad\phi\cdot\bL(\bx\oo)\grad\phi\,\dd\bz
=
\sum_{k=1}^n \alpha_k\int_B|\grad\phi\cdot\bee_k|^2\dd\bz,
\]
\n for all $\phi\in W_0^{1,2}(B)$.  The nonnegativity of the constant
eigenvalues $\alpha_k$, which yields $\bL(\bx\oo)$ positive semidefinite,
now follows from an appropriate
choice of $\phi$ (see, e.g., Truesdell \& Noll~\cite[\S68bis]{TN65} or
Dacorogna~\cite[p.~84]{Da08}).
\end{proof}
%
%

\subsection{The Elasticity Tensor} If the
elasticity tensor is uniformly positive definite at a deformation
$\bu\in\Def\cap\, W^{1,\infty}(\Om;\R^n)$, i.e.,
\be\label{eqn:ET-at-I-PD-2}
\bM:\CC\big(\bx,\bC_\bu(\bx)\big)[\bM]\ge 2\beta |\bM|^2,
\ee
\n for some $\beta>0$,  every $\bM\in\Symn$, and $a.e.~\bx\in\Omega$,
then the choice $\bM=\bB(\bx)$
together with an integration of
\eqref{eqn:ET-at-I-PD-2} yields
\be\label{eqn:SV-P2}
\int_\Omega \bB(\bx):\CC\big(\bx,\bC_\bu(\bx)\big)
\big[\bB(\bx)\big]\,\dd\bx
\ge
2\beta\int_\Omega |\bB(\bx)|^2\dd\bx.
\ee


\begin{lemma}\label{lem:PSV}  Let $\sigma$ satisfy (1)--(3) of
Hypothesis~\ref{def:W}.  Suppose that $\bu\in\AD$ satisfies
\eqref{eqn:SV-P2}, for some $\beta>0$ and all $\bB\in L^2(\Om;\Symn)$.
Moreover, assume that
\be\label{eqn:small+SV-in-Lem}
\bC_\bu(\bx)\in\sB \text{ for $a.e.~\bx\in \Omega$},
\ee
\n where $\sB$ is a nonempty, bounded, open set
with $\overline{\sB}\subset\Psymn$ and
$\bC_\bu:=(\grad\bu)^\rmT\grad\bu$.
Then there exists
an $\varepsilon>0$ such that
any $\bv\in\AD$ that satisfies, for $a.e.~\bx\in \Omega$,
\be\label{eqn:small+B-in-Lem}
\bC_\bv(\bx)\in\sB,
\qquad
\tsn{\bC_\bv-\bC_\bu}_{\BMO(\Omega)}
+
\Big|\dashint_\Omega (\bC_\bv-\bC_\bu)\,\dd\bx\Big|<\varepsilon,
\ee
\n $\bC_\bv:=(\grad\bv)^\rmT\grad\bv$, will also satisfy
\be\label{eqn:small+SV-in-Lem-2}
\int_\Omega \bE(\bx):\CC\big(\bx,\bC_\bv(\bx)\big)
\big[\bE(\bx)\big]\,\dd\bx
\ge
\beta\int_\Omega \big|\bE(\bx)\big|^2\,\dd\bx,
\quad \bE:=\bC_\bv-\bC_\bu.
\ee
\end{lemma}


Here and in the sequel, we use the notation
$\tsn{\bC_\bv-\bC_\bu}_{\BMO(\Omega)}$
to denote the $\BMO$-seminorm of the tensor $\bC_\bv-\bC_\bu$.
The definition is precisely as in \eqref{eqn:BMO-V-norm},
except one has the tensor in place of $\psi$ and the
Euclidean norm in place of the absolute value in the integral.

\begin{proof}[Proof of Lemma~\ref{lem:PSV}]
For clarity of exposition, we suppress the variable $\bx$.
Let $\bu\in\AD$ satisfy \eqref{eqn:SV-P2} and
 \eqref{eqn:small+SV-in-Lem} for
all $\bB\in L^2(\Om;\Symn)$.  Suppose that
$\bv\in\AD$ satisfies \eqref{eqn:small+B-in-Lem}
for some $\varepsilon>0$ to be determined and
define $\bE:=\bC_\bv-\bC_\bu$.
Then, Lemma~\ref{lem:taylor-III} with $\bV=\bC_\bv$, $\bU=\bC_\bu$,
and $\bL=\bC_\bv-\bC_\bu=\bE$  yields a constant
$\widehat{c}=\widehat{c}(\sB)>0$ such that,
for $a.e.~\bx\in \Omega$,
\be\label{eqn:taylor-app-in-Lem}
\bE:\CC\big(\bC_\bv\big)\big[\bE\big]
\ge \bE:\CC\big(\bC_\bu\big)
\big[\bE\big] - \widehat{c}|\bE|^3.
\ee


 If we now integrate
\eqref{eqn:taylor-app-in-Lem} over $\Omega$ and make use of
\eqref{eqn:SV-P2}  we find that
\be\label{eqn:E-taylor-in-Lem}
\int_\Omega
\bE:\CC\big(\bC_\bv\big)\big[\bE\big]\,\dd\bx
\ge
2\beta\int_\Omega|\bE|^2\,\dd\bx
-\widehat{c}\int_\Omega|\bE|^3\,\dd\bx.
\ee
\n We next note that inequality \eqref{eqn:RH} (with $q=3$ and $p=2$) of
Proposition~\ref{thm:main-2} yields a constant $J>0$ such
that, for the given $\bE=\bC_\bv-\bC_\bu$ that satisfies
\eqref{eqn:small+B-in-Lem}$_{2}$ and every $i,j\in \{1,\ldots, n\}$,
\be\label{eqn:quadratic>cubic-in-Lem}
\varepsilon J^3\int_\Omega
\big|E_{ij}\big|^2\dd\bx
\ge
\int_\Omega
\big|E_{ij}\big|^3\dd\bx.
\ee
\n Thus one deduces \eqref{eqn:small+SV-in-Lem-2} as a consequence of
\eqref{eqn:E-taylor-in-Lem} and \eqref{eqn:quadratic>cubic-in-Lem}
when $\varepsilon$ is sufficiently small.
\end{proof}


Finally, for future reference, we note that the uniform positivity of the
elasticity tensor is preserved under perturbations that are small in
 the space of strains.
We give a proof of this elementary result for the convenience of the reader.


\begin{lemma}\label{lem:temp-re-D2sigma-moved}
Let $\sigma$ satisfy
(1)--(3) of Hypothesis~\ref{def:W}.   Suppose that,
for some  $\bC\oo\in \Psymn$  and $\beta>0$,
\be\label{eqn:ET-at-I-PD-again-moved}
\bM:\CC(\bx,\bC\oo)[\bM]\ge 2\beta |\bM|^2,
\ee
\n for every $\bM\in\Symn$ and $a.e.~\bx\in\Omega$.
Then there exists an $\omega\oo\in (0,|\bC\oo|)$
such that any $\bC\in\Psymn$ that satisfies
\be\label{eqn:C-near-I-0-moved}
\big|\bC-\bC\oo\big|<\omega\oo
\ee
\n will also satisfy, for all $\bM\in\Symn$ and $a.e.~\bx\in\Om$,
\be\label{eqn:pos-A-sigma-moved}
\bM:\CC\big(\bx,\bC\big)[\bM]\ge \beta|\bM|^2.
\ee
\end{lemma}


\begin{proof}  Assume $\sigma$, $\CC$, $\bC\oo$, and $\beta$
 satisfy the hypotheses of
the Lemma. Define $\sB\subset\Psymn$
by
\[
\sB:=\{\bE\in\Psymn:|\bE-\bC\oo|<|\bC\oo|/2\}.
\]
\n Then Lemma~\ref{lem:taylor-III}, with $\bV=\bC$, $\bU=\bC\oo$, and
$\bL=\bM$, yields a constant $\widehat{c}>0$ such that, for $a.e.~\bx\in\Om$
and every  $\bM\in\Symn$,
\be\label{eqn:Taylor-A-sigma-moved}
\bM:\CC(\bx,\bC)[\bM]
\ge
\bM:\CC(\bx,\bC\oo)[\bM]
-\widehat{c}|\bM|^2|\bC-\bC\oo|.
\ee
\n The desired result, \eqref{eqn:pos-A-sigma-moved}, now follows from
\eqref{eqn:ET-at-I-PD-again-moved}, \eqref{eqn:C-near-I-0-moved}, and
\eqref{eqn:Taylor-A-sigma-moved} when $\omega\oo\le \beta/\widehat{c}$.
\end{proof}


\section{Uniqueness in \texorpdfstring{$\BMO\cap\, L^1$}{BMO}
Neighborhoods}\label{sec:unique-BMO+L1}

The first result of this section yields a comparison of the energy of an
equilibrium solution, $\bue$, to the energy of any admissible deformation
whose strains are sufficiently close to $\bC\e$ in $\BMO\cap\,L^1$.
Our main theorem then follows from this energy estimate.
It establishes that, given a solution $\bue$ of the
equilibrium equations whose principal stresses are positive
(or a smooth solution whose principal stresses are sufficiently small
and negative) and where the
integral of the elasticity tensor is uniformly positive, there is a
neighborhood of $\bC\e$ in $\Psymn$ in the $\BMO\cap\,L^1$-topology
in which there
are no other solutions of the equilibrium equations.


\begin{lemma}\label{lem:main-1} Let
$\sigma:\Omc\times\Psymn\to[0,\infty)$ satisfy (1)--(3) of
Hypothesis~\ref{def:W}.    Suppose that $\bu\e\in\AD$ is
a weak solution of the equilibrium equations, (\ref{eqn:EE}),
that satisfies, for some $k>0$,
some $\epsilon\in (0,1)$,
every $\bD\in L^2(\Om;\Symn)$,
and almost every $\bx\in\Om$,
\begin{gather}\label{eqn:CC-pos-def}
\int_\Omega \bD(\bx):\CC\big(\bx,\bC\e(\bx)\big)
\big[\bD(\bx)\big]\,\dd\bx
\ge
16k\int_\Omega |\bD(\bx)|^2\dd\bx,
\qquad \det\grad\bu\e(\bx)> \epsilon,
\end{gather}
\n where $\bC\e=\bC_{\bu\e}:=(\grad\bu\e)^\rmT\grad\bu\e$.
Fix $X\in\R$ with
$X>||\bC\e||_{\infty,\Om}$ and $X^{\mi1}<\epsilon$.  Then there
exists a $\altd=\altd(X)>0$ such that any $\bv\in\AD$ that
satisfies
\be\label{eqn:small+e-t-new-01}
\begin{gathered}
\tsn{\bC_\bv
-\bC\e}_{\BMO(\Omega)}
+
\Big|\dashint_\Omega \big[\bC_\bv -\bC\e\big]\,\dd\bx\Big|<\altd,
\\[4pt]
\|\bC_\bv\|_{\infty,\Om} <X,
\qquad
  \det\grad\bv>X^{\mi1}~a.e.
\end{gathered}
\ee
\n with
$\bC_\bv:=(\grad\bv)^\rmT\grad\bv$,
\n will also satisfy
\be\label{eqn:E-sigma-bound-zero-03}
\E(\bv)\ge \E(\bu\e) + k\int_\Omega|\bC_\bv-\bC\e|^2\dd\bx
+  \tfrac12\int_\Omega
\bK\big(\bx,\bC\e(\bx)\big):\bH^\rmT\bH\,\dd\bx,
\ee
\n where $\bH:=\grad\bv-\grad\bue$.
Moreover, if in addition $\bv=\bv\e$ is
a weak solution of the equilibrium equations, then
\be\label{eqn:E-sigma-bound-zero-04}
\E(\bu\e)\ge \E(\bv\e)
+ \frac{k}{2}\int_\Omega|\bC_{\bv\e}-\bC\e|^2\dd\bx
+\tfrac12 \int_\Omega
\bK\big(\bx,\bC_{\bv\e}(\bx)\big):\bH\e^\rmT\bH\e\,\dd\bx,
\ee
\n where $\bH\e:=\grad\bv\e-\grad\bue$.
\end{lemma}


\begin{theorem}\label{thm:new-final-1}  Let
$\sigma:\Omc\times\Psymn\to[0,\infty)$ satisfy (1)--(3) of
Hypothesis~\ref{def:W}.   Suppose that $\bu\e\in\AD$ is
a weak solution of the equilibrium equations that satisfies
(\ref{eqn:CC-pos-def}) and, for some $\kv\in\R$, every $\ba\in\R^n$,
and almost every $\bx\in\Om$,
\be\label{eqn:K-pos-def-at-ue-in-lem}
\ba\cdot\bK\big(\bx,\bC\e(\bx)\big)\ba\ge 2\ku |\ba|^2.
\ee
\n Assume in addition that either
\begin{enumerate}
\item[(a)] $\ku\ge0$; or
\item[(b)] $\bue\in C^1(\Omc;\R^n)$ and $\ku\ge-k C_M$,
where $C_M$ is given by Proposition~\ref{prop:CM15}.
\end{enumerate}
\n Fix $X\in\R$ with
$X>||\bC\e||_{\infty,\Om}$ and
$X^{\mi1}<\epsilon$  (see (\ref{eqn:CC-pos-def})).  Then there
exists a $\altd=\altd(X)>0$ such that any $\bv\in\AD$ that satisfies
(\ref{eqn:small+e-t-new-01})
will have strictly greater energy that $\bue$.  Moreover, if
$\bue$ and $\bv$ also satisfy, for some $\kv\in\R$,
every $\ba\in\R^n$, and almost every $\bx\in\Om$,
\[
\ba\cdot\Big[\bK\big(\bx,\bC\e(\bx)\big)+
\bK\big(\bx,\bC_{\bv}(\bx)\big)\Big]\ba\ge 2\kv |\ba|^2
\]
\n with either
\begin{enumerate}
\item[(i)] $\kv\ge0$; or
\item[(ii)] $2\kv\ge-3kC_M$ and at least one of
$\bue$ and $\bv$  is contained in  $C^1(\Omc;\R^n)$,
\end{enumerate}
\n then $\bv$ cannot be a weak solution of the equilibrium equations.
In particular if, for $a.e.~\bx\in\Om$,
$\bK(\bx,\bC\e(\bx))$ and $\bK(\bx,\bC_{\bv}(\bx))$
are positive semidefinite, then $\bv$ cannot be a
weak solution of the equilibrium equations.
\end{theorem}


\begin{remark}\label{rem:convexity-0}  1.~The inequality
\[
\bB:\CC\big(\bx,\bC\e(\bx)\big)[\bB]
\ge
 16k|\bB|^2,\  \text{ for all $\bB\in\Symn$ and $a.e.~\bx\in\Om$},
\]
\n implies that, for almost every $\bx\in\Om$, the function
$\bB\mapsto \sigma(\bx,\bB)$ is \emph{convex in a neighborhood of}
$\bC\e(\bx)$.
Gao, Neff, Roventa, \& Thiel~\cite{GNRT17} have shown that
the positivity of the principal stresses together with the assumption that,
for almost every $\bx\in\Om$, $\bC\e(\bx)$ is \emph{a (global)
point of convexity of} $\bB\mapsto \sigma(\bx,\bB)$, i.e.,
\[
\sigma(\bx,\bB)\ge \sigma\big(\bx,\bC\e(\bx)\big)
+\DD\sigma\big(\bx,\bC\e(\bx)\big):\big[\bB-\bC\e(\bx)\big]
\ \text{ for all   $\bB\in \Psymn$,}
\]
\n  implies that $\bu\e$ is an absolute minimizer of the energy.
2.~The conclusions of Lemma~\ref{lem:main-1} and
Theorem~\ref{thm:new-final-1} are valid under
slightly more general hypotheses.
It is clear that \eqref{eqn:CC-pos-def} need
not be satisfied by all $\bD\in L^2(\Om;\Symn)$, but only by all
$\bD$ that satisfy
\[
\bD=(\grad\bv)^\rmT\grad\bv-(\grad\bu)^\rmT\grad\bu
\]
\n for some $\bu,\bv\in\AD$.  For information on the
characterization of such mappings see, e.g.,  Blume~\cite{Bl89}
or  Ciarlet \&  Laurent~\cite{CL03}
and the references therein.
\end{remark}


\begin{proof}[Proof of Lemma~\ref{lem:main-1}] We suppress the
variable $\bx$ for clarity of exposition.
Let $\bu\e\in\AD$ be a weak solution of the equilibrium equations
 that  satisfies \eqref{eqn:CC-pos-def}.
Fix $X\in\R$ with
$X>||\bC\e||_{\infty,\Om}$ and $X^{\mi1}<\epsilon$.
Define $\sB\subset \Psymn$ by
\[
 \sB:= \{\bD\in\Psymn: X^{\mi2}<\det\bD,\ |\bD|<X\}.
\]
\n Let $\bv\in\AD$  satisfy \eqref{eqn:small+e-t-new-01}
for some $\altd>0$ to be determined.
Then \eqref{eqn:sPKS},  \eqref{eqn:ET-A}, and
Lemma~\ref{lem:taylor-III} with
$\bV=\bC_\bv$, $\bU=\bC\e$, and
$\bE=\bC_\bv-\bC\e$ yield a
constant $c=c(\sB)>0$ such that, for $a.e.~\bx\in\Om$,
\be\label{eqn:sigma-Taylor-1-lem}
\sigma(\bC_\bv)\ge\sigma(\bC\e)
+\tfrac12 \bE:\bK(\bC\e)
+\tfrac18\bE:\CC(\bC\e)[\bE]
-c|\bE|^3.
\ee
\n If we now integrate \eqref{eqn:sigma-Taylor-1-lem}
over $\Om$  and make use of \eqref{eqn:CC-pos-def} we find that
\be\label{eqn:sigma-Taylor-2-lem}
\int_\Omega \sigma(\bC_\bv)\,\dd\bx
\ge
\int_\Omega \sigma(\bC\e)\,\dd\bx
+
\tfrac12\int_\Omega\bE:\bK(\bC\e)\,\dd\bx
+
2k\int_\Omega |\bE|^2\dd\bx
-
c\int_\Omega|\bE|^3\,\dd\bx.
\ee


We next consider the term $\bE:\bK(\bC\e)$.  We make use
of an observation in \cite{GNRT17} (see, also, \cite{Sp82})
to write
\be\label{eqn:I-GNRT-lem}
\bE=\bC_\bv-\bC\e = (\grad\bu\e)^\rmT\bH
+ \bH^\rmT\grad\bu\e+\bH^\rmT\bH,
\qquad
\bH:=\grad\bv-\grad\bu\e.
\ee
\n Therefore,  \eqref{eqn:fPKS},  \eqref{eqn:I-GNRT-lem},
and the symmetry of $\bK$ gives us
\[
\begin{aligned}
\bE:\bK(\bC\e)&= 2\bH:\big[\grad\bu\e\bK(\bC\e)\big]
+ \bH^\rmT\bH:\bK(\bC\e)\\
&= 2\bH:\bS(\grad\bu\e)
+ \bH^\rmT\bH:\bK(\bC\e)
\end{aligned}
\]
\n and consequently
\be\label{eqn:B-dot-D-sigma-lem}
\begin{aligned}
\int_\Om\bE:\bK(\bC\e)\,\dd\bx
&=
2\int_\Om\bH:\bS(\grad\bu\e)\,\dd\bx
+ \int_\Om\bH^\rmT\bH:\bK(\bC\e)\,\dd\bx.
\end{aligned}
\ee
\n We then combine \eqref{eqn:sigma-Taylor-2-lem} and
\eqref{eqn:B-dot-D-sigma-lem} and make use of the
identity \eqref{eqn:Identity-ES} (which is a consequence of
the equilibrium equations \eqref{eqn:EE})
to conclude that
\be\label{eqn:E-sigma-bound-2-lem}
\E(\bv)\ge \E(\bu\e)
+2k \int_\Omega|\bE|^2\,\dd\bx
- c\int_\Omega|\bE|^3\,\dd\bx
+\tfrac12 \int_\Om\bH^\rmT\bH:\bK(\bC\e)\,\dd\bx.
\ee

Next, inequality \eqref{eqn:RH} (with $p=2$ and $q=3$) in
Proposition~\ref{thm:main-2} yields a constant $J>0$ such that,
for the given
$\bu\e$ and $\bv$ that satisfy \eqref{eqn:small+e-t-new-01} and
every $i,j\in\{1,2,\ldots,n\}$,
\be\label{eqn:quad-cubic}
\altd c J^3\int_\Omega|E_{ij}|^2\,\dd\bx
\ge
c\int_\Omega|E_{ij}|^3\,\dd\bx,
\ee
\n which together with \eqref{eqn:E-sigma-bound-2-lem} yields
the desired result, \eqref{eqn:E-sigma-bound-zero-03},
when $\delta$ is sufficiently small.

Finally, suppose that $\bv=\bv\e\not\equiv\bu\e$ is a
solution of the equilibrium equations.
Then the above
argument with $\bu\e$ replaced by $\bv\e$ and $\bv$ replaced by
$\bu\e$ yields
(see \eqref{eqn:sigma-Taylor-1-lem}--\eqref{eqn:quad-cubic})
the desired inequality, \eqref{eqn:E-sigma-bound-zero-04},
when $\delta$ is sufficiently small \emph{after an additional
observation:}
Lemma~\ref{lem:PSV} together with \eqref{eqn:CC-pos-def}
implies that the constant $k$ in \eqref{eqn:E-sigma-bound-zero-03}
becomes $k/2$ in \eqref{eqn:E-sigma-bound-zero-04}
(see \eqref{eqn:sigma-Taylor-1-lem} and \eqref{eqn:sigma-Taylor-2-lem}).
\end{proof}


\begin{proof}[Proof of Theorem~\ref{thm:new-final-1}]
We again suppress the variable $\bx$ for clarity of exposition.  Let
$\bu\e\in\AD$ be a weak solution of the equilibrium equations that
satisfies \eqref{eqn:CC-pos-def} and \eqref{eqn:K-pos-def-at-ue-in-lem}.
Then, in view of Lemma~\ref{lem:L-posdef} (with $\bL=\bK-2\ku\bI$),
\be\label{eqn:K-ge-tau-u}
\int_\Omega\bK\big(\bx,\bC\e(\bx)\big):\bH^\rmT\bH\,\dd\bx \ge
2\ku\int_\Omega|\bH|^2\,\dd\bx, \qquad \bH:=\grad\bv-\grad\bue.
\ee
\n Let $\delta>0$ be given by Lemma~\ref{lem:main-1} so that
any $\bv\in\AD$ that satisfies \eqref{eqn:small+e-t-new-01}
will also satisfy
\eqref{eqn:E-sigma-bound-zero-03}, that is,
\be\label{eqn:E-sigma-bound-zero-03-again}
\E(\bv)\ge \E(\bu\e) + k\int_\Omega|\bC_\bv-\bC\e|^2\dd\bx
+  \ku\int_\Omega|\bH|^2\,\dd\bx,
\ee
\n where we have made use of
\eqref{eqn:K-ge-tau-u}.


(a).~Clearly, $\ku\ge0$ yields $\E(\bv)\ge \E(\bu\e)$ since $k>0$.
Suppose that $\E(\bv)= \E(\bu\e)$.  Then $\bC_\bv=\bC\e~a.e.$  Note that,
in view of \eqref{eqn:CC-pos-def}$_2$,
\[
|\grad\bue(\bx)|^n\le \|\grad\bue\|^n_{\infty,\Om}
\le\bigg[\frac{\|\grad\bue\|^n_{\infty,\Om}}{\epsilon}\bigg]
\big[\det\grad\bue(\bx)\big] \
\text{ for $a.e.~\bx\in\Om$.}
\]
\n Consequently, Proposition~\ref{prop:Lo13} yields a rotation
$\bR\in\SO(n)$ such that $\grad\bv =\bR \grad\bue~a.e.$
Since $\Om$ is a connected open set,
$\bv=\bR\bu\e +\ba$ for some $\ba\in\R^n$. However, $\bv=\bu\e$
on $\sD$ and so $\bR=\bI$ and $\ba=\mathbf{0}$ since $\sD$ is
relatively open.  This establishes
the theorem under hypothesis (a).

(b).~Suppose now that $\bue\in C^1(\Omc;\R^n)$ and $\ku\ge-k C_M$,
where $C_M$ is given by Proposition~\ref{prop:CM15}.  Then
\eqref{eqn:E-sigma-bound-zero-03-again} together with
Proposition~\ref{prop:CM15} (with $p=q=2$) yields
\[
\E(\bv)\ge \E(\bu\e) +
kC_M\Big(\|\bv-\bu\e\|_{W^{1,2}(\Omega)}\Big)^2
-kC_M \int_\Omega|\grad\bv-\grad\bue|^2\dd\bx.
\]
\n Thus,  $\bv\not\equiv\bu\e$ satisfies $\E(\bv)>\E(\bu\e)$, as claimed.


Now suppose in addition that $\bv=\bv\e$ is a weak solution of the
equilibrium equations.  Then \eqref{eqn:E-sigma-bound-zero-03}
and \eqref{eqn:E-sigma-bound-zero-04} yield
\be\label{eqn:both-1}
0\ge
 3k\int_\Omega|\bC_{\bv\e}-\bC\e|^2\dd\bx
+  \int_\Omega
\Big[\bK\big(\bx,\bC\e(\bx)\big)
+\bK\big(\bx,\bC_{\bv\e}(\bx)\big)\Big]:\bH^\rmT\bH\,\dd\bx.
\ee
\n Next, in view of Lemma~\ref{lem:L-posdef}
(with $\bL=\bK(\bC\e)+\bK(\bC_{\bv\e})-2\kv\bI$)
\[
\int_\Omega
\Big[\bK\big(\bx,\bC\e(\bx)\big)
+\bK\big(\bx,\bC_{\bv\e}(\bx)\big)\Big]:\bH^\rmT\bH\,\dd\bx
\ge
2\kv\int_\Omega|\bH|^2\,\dd\bx,
\]
\n which together with  \eqref{eqn:both-1} gives us
\be\label{eqn:both}
0\ge
 3k\int_\Omega|\bC_{\bv\e}-\bC\e|^2\dd\bx
+2\kv \int_\Omega|\grad\bv\e-\grad\bue|^2\dd\bx.
\ee


(i).~If $\kv\ge0$, then $k>0$ yields $\bC_{\bv\e}=\bC\e~a.e.$ and
the same argument used to prove (a) now yields $\bv\e=\bue$.  Thus,
$\bv\not\equiv\bu\e$ cannot satisfy the equilibrium equations.

(ii).~Suppose now that $2\kv\ge-3kC_M$ and at least one of
$\bue$ and $\bv$  is contained in  $C^1(\Omc;\R^n)$.
Then \eqref{eqn:both} together with
Proposition~\ref{prop:CM15} (with $p=q=2$) now yields
\[
0\ge
 3k C_M\Big(\|\bv\e-\bu\e\|_{W^{1,2}(\Omega)}\Big)^2
-3kC_M \int_\Omega|\grad\bv\e-\grad\bue|^2\dd\bx.
\]
\n Therefore,  $\bv\e=\bue$ and hence
$\bv\not\equiv\bu\e$ cannot satisfy the equilibrium equations.
\end{proof}


\subsection{Deformations with Small Strain}\label{sec:Def-SS}

In this subsection we focus on deformations $\bu\in\AD$ whose
nonlinear \emph{Green-St.~Venant strain tensor}
\be\label{eqn:GSVST}
\bE_\bu(\bx):= \tfrac12\big[\bC_\bu(\bx)-\bI\big]
\ee
\n is sufficiently small.  Given one equilibrium solution $\bue$ whose
strain tensor $\bE\e$ is uniformly and sufficiently small we apply
Theorem~\ref{thm:new-final-1} to show that there is a
$\BMO\cap\,L^1$ neighborhood in strain space where there are no
other solutions of the equilibrium equations.


\begin{corollary}\label{cor:small-strain}  Let
$\sigma$ satisfy (1)--(3) of
Hypothesis~\ref{def:W}.   Suppose that, for some $k>0$,
\[
\bB:\CC(\bx,\bI)[\bB]\ge 32k |\bB|^2,
\]
\n for every $\bB\in\Symn$ and $a.e.~\bx\in\Omega$.
Assume further that $\bu\e \in \AD$ is a weak solution of the
equilibrium equations that satisfies, for some $\ku\in\R$,
 some $\epsilon\in(0,1)$,  every $\ba\in\R^n$,
and almost every $\bx\in\Om$,
\be\label{eqn:K-pd+det>0+Ce-small-again}
\begin{gathered}
\ba\cdot\bK\big(\bx,\bC\e(\bx)\big)\ba\ge 2\ku|\ba|^2,
\qquad \det\grad\bu\e(\bx)> \epsilon,  \qquad
\big\|\bC\e-\bI\big\|_{\infty,\Omega} < \omega\oo,
\end{gathered}
\ee
\n where $\bC\e=\bC_{\bu\e}:=(\grad\bu\e)^\rmT\grad\bu\e$,
$\omega\oo$ is the constant determined in
Lemma~\ref{lem:temp-re-D2sigma-moved} (with $\bC\oo=\bI$),
and either
\begin{enumerate}
\item[(a)] $\ku\ge0$; or
\item[(b)] $\bue\in C^1(\Omc;\R^n)$ and $\ku\ge-k C_M$,
where $C_M$ is given by Proposition~\ref{prop:CM15}.
\end{enumerate}
\n Fix $X\in\R$ with
$X>||\bC\e||_{\infty,\Om}$ and $X^{\mi1}<\epsilon$ and suppose that
$\bv\in\AD$  satisfies
\[
||\bC_\bv||_{\infty,\Om} <X,
\qquad  \det\grad\bv>X^{\mi1}~a.e.,
\]
\n with
$\bC_\bv:=(\grad\bv)^\rmT\grad\bv$.
Then there exists a
$\altd=\altd(X)>0$ such that if $\bu\e$ and $\bv$
satisfy
\be\label{eqn:small-Ce-and-Cv}
\begin{gathered}
\tsn{\bC\e}_{\BMO(\Omega)}  +
\Big|\dashint_\Omega \big[\bC\e -\bI\big]\,\dd\bx\Big|<\altd,
\\[2pt]
\tsn{\bC_\bv}_{\BMO(\Omega)}  +
\Big|\dashint_\Omega \big[\bC_\bv -\bI\big]\,\dd\bx\Big|<\altd,
\end{gathered}
\ee
\n or, merely,
\be\label{eqn:Cv-Ce-small}
\tsn{\bC_\bv - \bC\e}_{\BMO(\Omega)}  +
\Big|\dashint_\Omega \big[\bC_\bv -\bC\e\big]\,\dd\bx\Big|<2\altd,
\ee
\n then $\bv$ will have strictly greater energy that $\bue$.
Moreover, if
$\bue$ and $\bv$ also satisfy, for some $\kv\in\R$,
every $\ba\in\R^n$, and almost every $\bx\in\Om$,
\[
\ba\cdot\Big[\bK\big(\bx,\bC\e(\bx)\big)+
\bK\big(\bx,\bC_{\bv}(\bx)\big)\Big]\ba\ge 2\kv |\ba|^2
\]
\n with either
\begin{enumerate}
\item[(i)] $\kv\ge0$; or
\item[(ii)] $2\kv\ge-3kC_M$ and at least one of
$\bue$ and $\bv$  is contained in  $C^1(\Omc;\R^n)$,
\end{enumerate}
\n then $\bv$ cannot be a weak solution of the equilibrium equations.
In particular if, for $a.e.~\bx\in\Om$,
$\bK(\bx,\bC\e(\bx))$ and $\bK(\bx,\bC_{\bv}(\bx))$
are positive semidefinite, then $\bv$ cannot be a
weak solution of the equilibrium equations.
\end{corollary}


\begin{remark}\label{rem:BMO-constant-2} 1.~A simple computation shows
that, for any $\bH\in\Mn$,
\[
\tsn{\bC_\bv-\bH}_{\BMO(\Omega)}=\tsn{\bC_\bv}_{\BMO(\Omega)}.
\]
\n Thus, e.g., \eqref{eqn:small-Ce-and-Cv}$_2$ is the assumption
that  $\bC_\bv$
is close to the identity in $\BMO\cap\,L^1$.
2.~Note that Corollary~\ref{cor:small-strain}
does not require a stress-free reference configuration.
3.~Although hypothesis \eqref{eqn:K-pd+det>0+Ce-small-again}$_3$
forces $\bC\e$ to lie in an $L^\infty$-neighborhood of $\bI$, results
of \cite{FJM02} show that $\grad\bue$ then lies in a
$\BMO\cap\,L^1$-neighborhood of some rotation.
More precisely\footnote{This inequality follows directly from (4.3),
(4.4), and (5.14) in \cite{SS19}.}
there is a constant $D=D(\Om)$ such that for any $\bu\in\AD$
there exists a rotation $\bQ_\bu\in\SO(n)$ such that
\[
\tsn{\grad\bu}_{\BMO(\Om)}
+ \Big|\dashint_\Om \big[\grad\bu-\bQ_\bu\big]\,\dd\bx\Big|
\le
D\big\|\bC_\bu-\bI\big\|_{\infty,\Om}.
\]
\end{remark}


\begin{remark}\label{rem:lit-small}  Uniqueness of equilibrium in a
$\BMO$-neighborhood of a stress-free reference configuration was obtained
by John~\cite{Jo72} for the pure-displacement problem
(see also \cite{SS18}).  That result was recently extended to the
mixed problem in \cite{SS19}.  The object that is small in $\BMO$ in
these papers is the deformation gradient $\grad\bu$, rather than the strain
$\bC_\bu$.  Thus, the neighborhood in which there are no other solutions
is potentially larger in Corollary~\ref{cor:small-strain} than
in prior results.  However, our result requires the additional assumption
that each equilibria experience either tension or, at least,
 compressions that are sufficiently small.
\end{remark}


\begin{proof}[Proof of Corollary~\ref{cor:small-strain}]  We first note
that the triangle inequality together with \eqref{eqn:small-Ce-and-Cv}
yields \eqref{eqn:Cv-Ce-small}.  Thus we will
assume that $\bu\e$ and $\bv$ satisfy \eqref{eqn:Cv-Ce-small}.
We next observe that Lemma~\ref{lem:temp-re-D2sigma-moved}
yields an $\omega\oo>0$ such that any $\bE\in\Psymn$
with  $|\bE-\bI|<\omega\oo$ will satisfy
\be\label{eqn:pos-DD-sigma-again-x}
\bB:\CC(\bx,\bE)[\bB]\ge 16k|\bB|^2,
\ee
\n for all $\bB\in\Symn$ and $a.e.~\bx\in\Om$. Consequently,
\eqref{eqn:K-pd+det>0+Ce-small-again}$_3$
(together with \eqref{eqn:pos-DD-sigma-again-x})
yields
\[
\int_\Omega \bD(\bx):\CC\big(\bx,\bC\e(\bx)\big)
\big[\bD(\bx)\big]\,\dd\bx
\ge
16k\int_\Omega |\bD(\bx)|^2\dd\bx,
\]
\n for all $\bD\in L^2(\Om;\Symn)$.
Finally, we see that the hypotheses of
Theorem~\ref{thm:new-final-1} are satisfied, which then implies
the desired results.
\end{proof}


\section{Reference Configurations at Equilibrium}\label{sec:RCatE}

We here note that the statement of Theorem~\ref{thm:new-final-1}
simplifies when the
body in its reference configuration is itself at equilibrium.  Thus, we
assume that $\bue=\id\in C^1(\Omc;\R^n)$, i.e.,
\[
\bue(\bx)=\id(\bx) := \bx \ \text{ for }  \bx\in\Om,
\]
\n  where $\Om\subset\R^n$ is a Lipschitz domain.
Clearly, we also require that $\bd=\id$ on
$\sD$. However,
\emph{we do} \textbf{not} \emph{require that this
reference configuration be stress free.}

\begin{remark}\label{rem:not-equivalent} The above assumption
is akin to assuming that one is given
a body, $\overline\sB\subset\R^n$, and a mapping
$\bu:\overline\sB\to\R^n$ that is a solution of the equilibrium
equations and for which the deformed body
is a Lipschitz (or John) domain.  However, without further
assumptions on $\bu$ the two approaches are not equivalent.
In particular, $\bu(\partial\sB)$ need not be equal to
$\partial\Om$.
\end{remark}


\begin{theorem}[Theorem~\ref{thm:new-final-1}
for a Reference Configuration at Equilibrium]\label{thm:I-1} Let
$\sigma:\Omc\times\Psymn\to[0,\infty)$ satisfy (1)--(3) of
Hypothesis~\ref{def:W}.
 Suppose that $\bu\e=\id$
is a weak solution of the equilibrium equations, \eqref{eqn:EE}, that
satisfies, for some $k>0$,
every $\bD\in L^2(\Om;\Symn)$, every $\ba\in\R^n$,
and almost every $\bx\in\Om$,
\[
\int_\Omega \bD(\bx):\CC\big(\bx,\bI\big)
\big[\bD(\bx)\big]\,\dd\bx
\ge
16k\int_\Omega |\bD(\bx)|^2\dd\bx,\qquad
\ba\cdot\bK(\bx,\bI)\ba\ge -2kC_M |\ba|^2.
\]
\n Fix $X>\sqrt{n}$. Then there
exists a $\altd=\altd(X)>0$ such that any $\bv\in\AD$ that
satisfies $\bv\ne\id$,
\[
\tsn{\bC_\bv}_{\BMO(\Omega)}
+
\Big|\dashint_\Omega \big[\bC_\bv-\bI\big]\,\dd\bx\Big|<\altd, \qquad
\|\bC_\bv\|_{\infty,\Om} <X,
\qquad
  \det\grad\bv>X^{\mi1}~a.e.,
\]
\n with $\bC_\bv:=(\grad\bv)^\rmT\grad\bv$,
will have strictly greater energy
than  $\id$.   Moreover, if in addition,
for almost every $\bx\in\Om$ and every $\ba\in\R^n$,
\[
\ba\cdot\Big[\bK\big(\bx,\bI\big)+
\bK\big(\bx,\bC_{\bv}(\bx)\big)\Big]\ba\ge -3kC_M |\ba|^2
\]
\n then $\bv$ cannot be a weak solution of the equations of equilibrium.
In particular if, for $a.e.~\bx\in\Om$,
$\bK(\bx,\bI)$ and $\bK(\bx,\bC_{\bv}(\bx))$
are positive semidefinite, then $\bv$ cannot be a
weak solution of the equilibrium equations.
\end{theorem}


\begin{remark}\label{rem:better}  (1).~A slightly better result can
be obtained by replacing Proposition~\ref{prop:CM15} with the result
it is based upon in \cite{FJM02}. (2).~We note that statements of
prior results of ours from \cite{SS19} as well as a prior result of
J.~Sivaloganathan and one of the current authors from \cite{SS18}
also simplify in the special case when $\bue=\id$.
\end{remark}


\appendix

\section{Taylor's Theorem}\label{sec:Taylor}

 The following result is a consequence of Taylor's theorem
(see, e.g., \cite[Section~4.6]{Ze86} and \cite[Appendix~A]{SS19}).


\begin{lemma}\label{lem:taylor-III} Let
$\sigma:\overline{\Omega}\times\Psymn\to\R$ be as given in
(1)--(3) of Hypothesis~\ref{def:W}.
Suppose that $\sB\subset\Psymn$ is a nonempty,
bounded, open set that satisfies $\overline{\sB}\subset\Psymn$.
Then there exists constants $c=c(\sB)>0$ and  $\widehat{c}=\widehat{c}(\sB)>0$
such that, for every $\bU,\bV\in\overline{\sB}$, $\bL\in\Symn$,
and almost every $\bx\in\Omega$,
\[
\begin{gathered}
\sigma(\bx,\bV)\ge \sigma(\bx,\bU) + \bE:\DD\sigma(\bx,\bU)
+\tfrac12\bE:\DD^2 \sigma(\bx,\bU)[\bE] -c|\bE|^3,\\[4pt]
\bL:\DD^2 \sigma(\bx,\bV)[\bL]
\ge
\bL:\DD^2 \sigma(\bx,\bU)[\bL] -\widehat{c}|\bV-\bU||\bL|^2,
\end{gathered}
\]
\n where $\bE:=\bV-\bU$.
\end{lemma}


\end{document}